\documentclass[a4paper]{amsart}

\usepackage{microtype}

\usepackage{amsmath,amstext,amssymb,mathrsfs,amscd,amsthm,indentfirst, bbm}
\usepackage{amsfonts}

\usepackage{booktabs}
\usepackage{verbatim}
\usepackage{enumerate}
\usepackage{graphicx}
\usepackage{textcomp}
\usepackage{mathtools}

\usepackage{tikz, tikz-cd}
\tikzset{shorten <>/.style={shorten >=#1,shorten <=#1}}
\usetikzlibrary{matrix,calc,arrows}

\usepackage{spectralsequences}

\newcommand{\bA}{\mathbb{A}}

\newcommand{\bC}{\mathbb{C}}

\newcommand{\bF}{\mathbb{F}}

\newcommand{\bP}{\mathbb{P}}
\newcommand{\bQ}{\mathbb{Q}}
\newcommand{\bR}{\mathbb{R}}

\newcommand{\bZ}{\mathbb{Z}}

\newcommand\lra{\longrightarrow}

\newcommand\Diff{\mathrm{Diff}}

\newcommand\Coker{\operatorname*{Coker}}

\newcommand{\hcoker}{/\!\!/}

\newcommand{\map}{\mathrm{map}}

\renewcommand{\epsilon}{\varepsilon}

\newcommand{\Sq}{\mathrm{Sq}}

\newcommand{\Mod}{\mathrm{Mod}}

\newcommand{\SO}{\mathrm{SO}}
\newcommand{\OO}{\mathrm{O}}
\newcommand{\U}{\mathrm{U}}

\newcommand{\Mon}{\mathrm{Mon}}
\newcommand{\MCG}{\mathrm{MCG}}

\newcommand{\Stab}{\mathrm{Stab}}

\newcommand{\sfr}{\mathrm{sfr}}
\newcommand{\hyp}{\mathrm{hyp}}

\mathchardef\ordinarycolon\mathcode`\:
\mathcode`\:=\string"8000
\begingroup \catcode`\:=\active
  \gdef:{\mathrel{\mathop\ordinarycolon}}
\endgroup

\theoremstyle{plain}
\newtheorem{MainThm}{Theorem}

\newtheorem{theorem}{Theorem}[section]

\newtheorem{proposition}[theorem]{Proposition}
\newtheorem{lemma}[theorem]{Lemma}

\newtheorem{corollary}[theorem]{Corollary}

\theoremstyle{definition}

\theoremstyle{remark}

\newtheorem{remark}[theorem]{Remark}
\newtheorem*{remark*}{Remark}

\newtheorem*{example*}{Example}

\numberwithin{equation}{section}

\usepackage[hypertexnames=false]{hyperref}

\title[Monodromy and mapping class groups]{Monodromy and mapping class groups\\ of 3-dimensional hypersurfaces}

\author{Oscar Randal-Williams}
\email{o.randal-williams@dpmms.cam.ac.uk}
\address{Centre for Mathematical Sciences\\
Wilberforce Road\\
Cambridge CB3 0WB\\
UK}
\begin{document}
\begin{abstract}
We describe the subgroup of the mapping class group of a hypersurface in $\mathbb{CP}^4$ consisting of those diffeomorphisms which can be realised by monodromy.
\end{abstract}
\maketitle

\setcounter{tocdepth}{1}
\tableofcontents

\section{Introduction}

Let $d$ be a positive integer, and $X_d \subset \bC\bP^4$ denote the degree $d$ Fermat hypersurface. It resides in the universal family $\mathcal{X}_d \to \mathcal{U}_d$ of smooth 3-dimensional degree $d$ hypersurfaces, whose base is the open subspace $\mathcal{U}_d \subset \mathbb{P}H^0(\bC\bP^4 ; \mathcal{O}(d))$ of degree $d$ homogeneous polynomials in 5 variables whose zero locus is smooth. We write
$$\Mon_d := \pi_1(\mathcal{U}_d, X_d)$$
for the monodromy group of this family.

On the other hand $X_d$ can be considered as an oriented 6-manifold, and we write
$$\MCG_d := \pi_0 \Diff^+(X_d)$$
for the oriented mapping class group of this manifold: the group of isotopy classes of orientation-preserving diffeomorphisms. Fibre transport for the universal family yields a group homomorphism
\begin{equation*}
\alpha: \Mon_d \lra \MCG_d.
\end{equation*}
A presentation for the group $\Mon_d$ has been given by L{\"o}nne \cite{Lonne}, who suggests \cite[p.\ 362]{Lonne} studying the map $\alpha$. Furthermore, a quite complete description of $\MCG_d$ has been given by Kreck and Su \cite{KSv3}. Using the latter, Hain \cite{HainMCGKahler} has explained how the methods of Sullivan \cite{sullivaninf} show that the image of $\alpha$ has infinite index in $\MCG_d$. Our goal is to completely describe the image of $\alpha$, in terms of Kreck and Su's description of $\MCG_d$.

Most of the answer will be described in terms of automorphisms of $\pi_3(X_d)$ respecting certain structures. There is a natural extension
$$0 \lra \bZ/d \lra \pi_3(X_d) \overset{h}\lra H_3(X_d;\bZ) \lra 0,$$
where $h$ is the Hurewicz map, and the $\bZ$-valued intersection form of $X_d$ induces an antisymmetric form $\lambda$ on $\pi_3(X_d)$, whose radical is $\bZ/d \leq \pi_3(X_d)$. We write $\mathrm{Aut}(\pi_3(X_d), \lambda)$ for the group of automorphisms of $\pi_3(X_d)$ which preserve $\lambda$ \emph{and which induce the identity on the radical $\bZ/d \leq \pi_3(X_d)$ of $\lambda$}. We will explain that $\lambda$ has a quadratic refinement
$$\mu : \pi_3(X_d) \lra \bZ/2,$$
induced by the embedding $X_d \subset \bC\bP^4$, and that this quadratic refinement is $\Mon_d$-invariant. By a quadratic refinement of $\lambda$ we mean that $\mu$ is a function satisfying
$$\mu(a+b) = \mu(a) + \mu(b) + \lambda(a,b) \mod 2.$$
We write $\mathrm{Aut}(\pi_3(X_d), \lambda, \mu)$ for the subgroup of $\mathrm{Aut}(\pi_3(X_d), \lambda)$ consisting of those automorphisms which preserve $\mu$. 

\begin{MainThm}\label{thm:Main} Suppose $d \geq 3$.
\begin{enumerate}[(i)]
\item The map $\Mon_d \to \mathrm{Aut}(\pi_3(X_d), \lambda, \mu)$ is surjective.

\item The kernel of the surjection $\mathrm{Im}(\alpha) \to \mathrm{Aut}(\pi_3(X_d), \lambda, \mu)$ consists precisely of those isotopy classes of diffeomorphisms of $X_d$ which may be represented by diffeomorphisms supported in an embedded 6-disc.%up to isotopy may be supported in a disc. 
\end{enumerate}
\end{MainThm}

The latter is a quotient of $\pi_0(\Diff_\partial(D^6)) = \Theta_7 = \bZ/28$, and may be extracted from the work of Kreck and Su (see Lemma \ref{lem:WhatIsKd}). %Writing $G_d := \mathrm{Ker}(\Theta_7 \to \MCG_d)$ i
It is given by
$$\frac{\Theta_7}{\mathrm{Ker}(\Theta_7 \to \MCG_d)} \cong \begin{cases}
0 & d \equiv 2,4,6,10,12,14 \mod 16\\
\bZ/2 & d \equiv 3,5,8,11,13 \mod 16\\
\bZ/4 & d \equiv 0, 1, 7, 9, 15 \mod 16
\end{cases}
\oplus \begin{cases}
0 & d \not \equiv 0 \mod 7\\
\bZ/7 & d  \equiv 0 \mod 7.
\end{cases}$$
There is therefore a central extension
$$0 \lra \Theta_7/\mathrm{Ker}(\Theta_7 \to \MCG_d) \lra \mathrm{Im}(\alpha) \lra \mathrm{Aut}(\pi_3(X_d), \lambda, \mu) \lra 0.$$
In Section \ref{sec:Wg1} we shall combine Theorem \ref{thm:Main} with existing work on diffeomorphism groups of the manifolds $W_{g,1} := (\#^g S^3 \times S^3) \setminus \mathrm{int}(D^6)$ to partially understand this extension, and in particular obtain the following, where we recall that the \emph{finite residual} of a group is the intersection of all of its finite-index subgroups.

\begin{MainThm}\label{thm:B}
The finite residual of $\mathrm{Im}(\alpha)$ is the subgroup $\Theta_7/\mathrm{Ker}(\Theta_7 \to \MCG_d)$.% $\mathrm{Ker}(\mathrm{Im}(\alpha) \to \mathrm{Aut}(\pi_3(X_d), \lambda, \mu))$.
\end{MainThm}
It follows that $\mathrm{Im}(\alpha)$ is often not residually finite, so neither is $\MCG_d$.

\begin{remark}
We use the assumption $d \geq 3$ to guarantee that $X_d$ contains a modest number of $S^3 \times S^3$ connect-summands ($d \geq 3$ implies that it contains at least 5). If $d=1$ or $2$ then $X_d$ contains no $S^3 \times S^3$ connect summands. The work of Kreck and Su shows that $\MCG_1 \cong \bZ/4$ (see also \cite[Remark II.11]{Brumfiel}) and $\MCG_2 = 0$. On the other hand $\Mon_1 = 0$ because \emph{all} degree 1 hypersurfaces are smooth, so $\mathrm{Im}(\alpha)$ is trivial if $d \leq 2$.
\end{remark}

The map $\alpha$ naturally factors through the symplectic mapping class group
$$\alpha: \Mon_d \lra \pi_0 \mathrm{Symp}(X_d) \lra \MCG_d.$$
This work began in discussions with Ailsa Keating and Ivan Smith, exploring the possibility of using this factorisation to investigate $\pi_0 \mathrm{Symp}(X_d)$. While $\mathrm{Im}(\alpha)$ clearly gives a lower bound for $\mathrm{Im}(\pi_0 \mathrm{Symp}(X_d) \to \MCG_d)$, trying to go any further seemed to produce more questions than answers on the symplectic side:

\begin{enumerate}[(i)]
\item Is the quadratic refinement $\mu$ invariant under the action of $\pi_0 \mathrm{Symp}(X_d)$?

\item Does the ``distortion of the first Pontrjagin class'' vanish for symplectomorphisms? (See Section \ref{sec:Distorsion} for this notion.)

\item Do $\Mon_d$ and $\pi_0 \mathrm{Symp}(X_d)$ have the same image in $\MCG_d$?

\item Is $\Mon_d \to \pi_0 \mathrm{Symp}(X_d)$ surjective?
\end{enumerate}
If the answer to (iv) is ``yes'' then it is for the other questions too, but I am told it is not currently accessible. With regards to (i), $\mu$ is not $\MCG_d$-invariant, and neither is it invariant for the subgroup of $\MCG_d$ which preserves the almost complex structure of $X_d$ (because this is the whole of $\MCG_d$, see Remark \ref{rem:MCGPresUnstableACStr}).

\vspace{2ex}

\noindent\textbf{Strategy.} The strategy for identifying the image of $\alpha : \Mod_d \to \MCG_d$ is fairly clear: one must produce constraints on mapping classes---which are satisfied by monodromy---to cut down the image, and then one must then show that enough mapping classes are realised by monodromy to see that the constraints are sharp. This is broadly what we will do.

The first constraint, developed in Section \ref{sec:Upper1}, is based on the observation that the universal family of smooth hypersurfaces over $\mathcal{U}_d$ is equipped with a certain (stable) tangential structure $\ell^\hyp_{X_d}$, so that $\alpha$ lands in the stabiliser $\mathrm{Stab}_{\MCG_d}(\ell^\hyp_{X_d})$ of this tangential structure. This stabiliser can be analysed following Krannich \cite{KrannichMCG} or Kupers and the author \cite{KR-WFram}, and in particular it is shown to preserve a quadratic refinement $\mu$ on $\pi_3(X_d)$, which with the work of Kreck and Su \cite{KSv3} leads to a central extension
$$0 \lra \{\substack{\text{diffeomorphisms supported} \\ \text{near $S^2 \subset X_d$}}\} \lra\mathrm{Stab}_{\MCG_d}(\ell^\hyp_{X_d}) \lra \mathrm{Aut}(\pi_3(X_d), \lambda, \mu) \lra 0.$$
To explain the left-hand term, we choose once and for all an embedding $S^2 \hookrightarrow X_d$ whose homotopy class represents a generator of $\pi_2(X_d) \cong \bZ$: the left-hand term then consists of those isotopy classes of diffeomorphisms which may be represented by diffeomorphisms supported in a tubular neighbourhood of this embedded $S^2$. This subgroup 
%The structure of the subgroup of diffeomorphisms supported near $S^2 \subset X_d$, which is a sphere generating $\pi_2(X_d) \cong \bZ$, 
has been completely determined by Kreck and Su. In particular it is finite and abelian. Choosing an embedded 6-disc inside the tubular neighbourhood of $S^2 \subset X_d$ 
%Implanting those diffeomorphisms which are supported on a disc 
gives a homomorphism
$$\Phi : \pi_0(\Diff_\partial(D^6)) = \Theta_7 \lra \{\substack{\text{diffeomorphisms supported} \\ \text{near $S^2 \subset X_d$}}\}.$$

The second constraint, developed in Section \ref{sec:FurtherUpperBounds}, is the construction of a surjective map $\kappa : \mathrm{Stab}_{\MCG_d}(\ell^\hyp_{X_d}) \to \mathrm{Coker}(\Phi)$ such that $\kappa \circ \alpha$ is trivial. This map is obtained indirectly, essentially by calculating the abelianisation of $\mathrm{Stab}_{\MCG_d}(\ell^\hyp_{X_d})$. This is done by relating this stabiliser to a certain moduli space of manifolds (diffeomorphic to $X_d$ and equipped with a certain tangential structure), describing the first homology of the latter in terms of Thom spectra using the work of Galatius and the author \cite{grwstab2}, and then calculating using various techniques from stable homotopy theory.

We get lower bounds for the image of $\alpha : \Mon_d \to \MCG_d$ from three sources. Firstly, work of Krylov \cite{Krylov} quickly shows that the subgroup $\Phi(\Theta_7)$ can be realised by monodromy. Secondly, work of Beauville \cite{Beauville} (for which we supply a missing detail) fully describes the automorphisms of $H_3(X_d;\bZ)$ that can be realised by monodromy. Finally, in Section \ref{sec:LowerBounds} we explain how work of Pham \cite{Pham} and of Looijenga \cite{LooijengaFermat} concerning the action of groups of $d$th roots of unity on $X_d$ can be used to extend Beauville's work to fully describe the automorphisms of $\pi_3(X_d)$ which can be realised by monodromy.

\vspace{2ex}

\noindent\textbf{Acknowledgements.} I am grateful to A.~Keating and I.~Smith for many useful discussions surrounding this work and their feedback on earlier versions, to R.~Hain and D.~Ranganathan for comments, and to the referee for their feedback. I was supported by the ERC under the European Union’s Horizon 2020 research and innovation programme (grant agreement No.\ 756444).

\section{Some recollections}

\subsection{Algebraic topology of hypersurfaces}\label{sec:TopOfHyp}

Let us explain our notation for the cohomology of a degree $d$ smooth hypersurface $X \subset \bC\bP^4$. Writing $x \in H^2(X;\bZ)$ for the restriction from $\bC\bP^4$ of the hyperplane class, by Poincar{\'e} duality there is a unique class $y \in H^4(X;\bZ)$ such that $\int_X x \cdot y=1$. As $\int_X x^3 = d$, it follows that $x^2 = d  y$. Then the even cohomology  of $X$ is
$$H^0(X;\bZ) = \bZ\{1\} \quad H^2(X;\bZ) = \bZ\{x\} \quad H^4(X;\bZ) = \bZ\{y\} \quad H^6(X;\bZ) = \bZ\{x y\}.$$
The only odd cohomology is $H^3(X;\bZ)$, which is equipped with the nondegenerate antisymmetric form $(a,b) \mapsto  \int_X a \cdot b$.

The inclusion $X \subset \bC\bP^4$ is 3-connected by the Lefschetz hyperplane theorem. Combined with the fact that $\pi_3(\bC\bP^4)=\pi_4(\bC\bP^4)=0$ and the Hurewicz theorem, this gives isomorphisms
$$\pi_3(X) \overset{\sim}\longleftarrow \pi_4(\bC\bP^4, X) \overset{\sim}\lra H_4(\bC\bP^4, X ; \bZ).$$
The long exact sequence on homology for this pair takes the form
$$\cdots \lra H_4(X;\bZ) \lra H_4(\bC\bP^4;\bZ) \lra H_4(\bC\bP^4, X ; \bZ) \lra H_3(X;\bZ) \lra 0$$
and so gives an extension
\begin{equation}\label{eq:pi3defn}
0 \lra \bZ/d\{\eta\} \lra \pi_3(X) \overset{h}\lra H_3(X;\bZ) \lra 0.
\end{equation}
The subgroup $\bZ/d$ is generated by the Hopf map $\eta : S^3 \to S^2$ composed with the generator of $\pi_2(X) \cong H_2(X;\bZ) \cong \bZ$ dual to $x \in H^2(X;\bZ)$. The intersection form of $X$ induces an antisymmetric form $\lambda$ on $\pi_3(X)$, whose radical is precisely the subgroup $\bZ/d\{\eta\}$.

The definition of $X$ as a hypersurface gives an isomorphism of complex vector bundles $TX \oplus \mathcal{O}(d)\vert_X \oplus \underline{\bC} = \mathcal{O}(1)^{\oplus 5}\vert_X$, and so its total Chern class is
\begin{align*}
c(TX) = \tfrac{(1+x)^5}{(1+d  x)} &= 1 + (5 - d) x + (d^2 - 5 d + 10) x^2 + (-d^3 + 5 d^2 - 10 d + 10) x^3\\
&= 1 + (5 - d) x + (d^2 - 5 d + 10) d y + (-d^3 + 5 d^2 - 10 d + 10) d x y.
\end{align*}
As the third Chern class of $TX$ is also the Euler class of this vector bundle, we see that $\chi(X) = (-d^3 + 5 d^2 - 10 d + 10) d$ and so
$$\mathrm{rank} \, H^3(X;\bZ) = 4 -  (-d^3 + 5 d^2 - 10 d + 10) d = d^4-5d^3+10d^2-10d + 4.$$
We write $2g$ for this number: by a theorem of Wall \cite[Theorem 1]{WallV} there is a decomposition $X \cong X' \# g(S^3 \times S^3)$ for some 6-manifold $X'$; it will be useful to know that $g \geq 5$ as long as $d \geq 3$. The class $c_1(TX) = (5-d) x$ reduces modulo 2 to $w_2(TX)$, so $X$ is Spin if and only if $d$ is odd. Its first Pontrjagin class is 
\begin{align*}
p_1(TX) &= - c_2(TX \otimes \bC) = -c_2 (TX \oplus \overline{TX}) = -(2c_2(TX) - c_1(TX)^2)\\
&= (5-d^2)d y.
\end{align*}

\subsection{Monodromy}

Beauville \cite{Beauville} has studied the monodromy action of the universal family of degree $d$ hypersurfaces on the middle cohomology, and his results in particular apply to the composition
\begin{equation}\label{eq:Beauville}
\Mon_d \overset{\alpha}\lra \MCG_d \lra \mathrm{Aut}(H^3(X_d;\bZ), (a,b) \mapsto  a\cdot b).
\end{equation}
He explains that the collection $\Delta_{X_d} \subset H^3(X_d;\bZ)$ of vanishing cycles for a Lefschetz fibration containing $X_d$ forms a ``réseau évanescent'', and that for $d \geq 3$ $X_d$ admits a deformation to a $E_6$-singularity, so \cite[Théorème 3]{Beauville} applies to say that \eqref{eq:Beauville} is either 
\begin{enumerate}[(i)]
\item surjective, or 
\item has image the stabiliser of some quadratic refinement of $- \cdot -$.
\end{enumerate}
It is an easy algebraic exercise to see that distinct quadratic refinements have distinct stabilisers, so in the latter case there is a canonical quadratic refinement $q_{X_d} : H^3(X_d;\bZ) \to \bZ/2$ associated to $X_d$.

Beauville then argues that these two cases correspond to $d$ being even or odd respectively. For $d$ even he shows that there is indeed no $\Mon_d$-invariant quadratic refinement, so case (i) holds. For $d$ odd he refers to work of Browder \cite{Browder} or Wood \cite{Wood} for the existence of a quadratic refinement, however we are concerned with dimension 3 and in the exceptional dimensions $1$, $3$, and $7$ those results do not apply\footnote{This is explicit in \cite[Theorem B]{Browder}. On the other hand \cite[Theorem 2]{Wood} calculates the Kervaire invariant of $X_d$ without seeming to concern itself with whether this invariant is defined, but the definition at the end of \S 1 defines the zero quadratic function in dimensions 1, 3, or 7. However, see Remark \ref{rem:QuadRefExists}.}. In fact, as Beauville explains \cite[p.\ 15]{Beauville} there is no $\MCG_d$-invariant quadratic refinement, so any construction of $q_{X_d}$ cannot be by plain differential topology. Nonetheless Beauville is correct that case (ii) holds when $d$ is odd: one contribution of this paper is to produce, in Section \ref{sec:SurgeryKernel}, the required $\Mon_d$-invariant quadratic refinement. Pending this detail, we record Beauville's result as follows.

\begin{theorem}[Beauville \cite{Beauville}]\label{thm:Beauville}
For $d \geq 3$ the map
$$\Mon_d \lra \begin{cases}
\mathrm{Aut}(H^3(X_d;\bZ), (a,b) \mapsto a \cdot b, q_{X_d}) & d \text{ odd}\\
\mathrm{Aut}(H^3(X_d;\bZ), (a,b) \mapsto a \cdot b) & d \text{ even}
\end{cases}$$
is surjective.\qed
\end{theorem}

\subsection{The mapping class group}

Kreck and Su \cite{KSv3} have made a detailed study of the mapping class groups of 6-manifolds which ``look like 3-dimensional complete intersections'': of course a 3-dimensional hypersurface such as $X_d$ is such a manifold. For us the most useful form of their result is as follows \cite[Theorem 2.6]{KSv3}. There are extensions
\begin{equation}\label{eq:KS}
\begin{tikzcd}[row sep=small]
1 \rar & \mathrm{SMCG}_d \rar & \MCG_d \rar & \mathrm{Aut}(\pi_3(X_d), \lambda) \rar & 1\\
1 \rar & \mathrm{K}_d \rar & \mathrm{SMCG}_d \rar{v_p} & H^3(X_d ; \bZ) \rar & 1.
\end{tikzcd}
\end{equation}
The first defines the subgroup $\mathrm{SMCG}_d$, and expresses the fact that all automorphisms of $\pi_3(X_d)$ which preserve $\lambda$ are are the identity on its radical $\bZ/d\{\eta\}$ can be realised by diffeomorphisms. The second describes $\mathrm{SMCG}_d$ as an extension of $H^3(X_d ; \bZ)$ by a certain finite abelian group $\mathrm{K}_d$ (we shall describe the map $v_p$ later: experts will understand the idea from the phrase ``distortion of the first Pontrjagin class''). To describe the subgroup $\mathrm{K}_d$ we choose once and for all an embedding $S^2 \hookrightarrow X_d$ whose homotopy class generates $\pi_2(X_d) \cong \bZ$ (the isotopy class of this embedding is unique by a theorem of Haefliger \cite{Haefliger}): then $\mathrm{K}_d$ is the subgroup of those isotopy classes of diffeomorphisms which may be represented by diffeomorphisms supported in a tubular neighbourhood of this embedded sphere. 
%The subgroup $\mathrm{K}_d$ is generated by those diffeomorphisms of $X_d$ that can be supported in a neighbourhood of $S^2 \subset X_d$. 
This includes those isotopy classes of diffeomorphisms which may be represented by diffeomorphisms supported in an embedded 6-disc, which defines a homomorphism $\Phi : \pi_0(\Diff_\partial(D^6)) \cong \Theta_7 \cong \bZ/28 \to \mathrm{K}_d$. The group $\mathrm{K}_d$ has been completely determined by Kreck and Su, and we summarise what we need from their calculations as follows.

\begin{lemma}[Kreck--Su \cite{KSv3}]\label{lem:WhatIsKd}
The map $\Phi : \Theta_7 \to \mathrm{K}_d$ has
$$\mathrm{Ker}(\Phi) \cong \begin{cases}
\bZ/4 & d \equiv 2,4,6,10,12,14 \mod 16\\
\bZ/2 & d \equiv 3,5,8,11,13 \mod 16\\
0 & d \equiv 0, 1, 7, 9, 15 \mod 16
\end{cases} \oplus \begin{cases}
\bZ/7 & d \not\equiv 0 \mod 7\\
0 & d \equiv 0 \mod 7
\end{cases}$$
and
$$\mathrm{Coker}(\Phi) \cong \begin{cases}
\bZ/2 & d \equiv 0 \mod 4\\
0 & d \not\equiv 0 \mod 4
\end{cases} \oplus \begin{cases}
\bZ/3 & d \equiv 0 \mod 3\\
0 & d \not\equiv 0 \mod 3.
\end{cases}$$
\end{lemma}
\begin{proof}[Proof sketch]
The bulk of this concerns the 2-torsion, which in the $d$ odd case, for example, requires a case-by-case analysis of table (8.1) on p.\ 42 of \cite{KSv3} reduced modulo 4, with $k := \tfrac{1}{4}(5-d^2)d$. Looking at the 4th and 5th columns in particular, one sees that the quotient of $(\bZ/4)^3$ by the relations given by the columns of this table is generated by the class of the first basis vector. The 7-torsion is immediate from the discussion just before this table, as is the 3-torsion. The $d$ even case is analogous, using the table on p.\ 51 of \cite{KSv3}.
\end{proof}

\section{Upper bounds: tangential structures and quadratic forms}\label{sec:Upper1}

The goal of this section is as follows. We shall describe a certain tangential structure $\theta^{\hyp}$ and a $\theta^{\hyp}$-structure $\ell^{\hyp}_{X_d}$ on $X_d$, which is preserved up to homotopy by monodromy: thus there is a factorisation
$$\alpha : \Mon_d \lra \Stab_{\MCG_d}(\ell^{\hyp}_{X_d}) \leq  \MCG_d.$$
We will then use the work of Kreck and Su to analyse the subgroup $\Stab_{\MCG_d}(\ell^{\hyp}_{X_d})$, and show that the extensions \eqref{eq:KS} simplify to a single central extension
\begin{equation*}
\begin{tikzcd}[row sep=small]
1 \rar & \mathrm{K}_d \rar & \Stab_{\MCG_d}(\ell^{\hyp}_{X_d}) \rar & \mathrm{Aut}(\pi_3(X_d), \lambda, \mu) \rar & 1,
\end{tikzcd}
\end{equation*}
for $\mu : \pi_3(X_d) \to \bZ/2$ a certain quadratic refinement which we will describe.

\subsection{Tangential structures}\label{sec:TS}
A \emph{stable tangential structure} is a Serre fibration $\bar{\theta} : B \to B\OO$, and a \emph{$\bar{\theta}$-structure} on a manifold $M$ is a choice of lift $\ell : M \to B$ of the map $\tau^s_M : M \to B\OO$ classifying the stable tangent bundle of $M$. Familiar examples are $\bar{\theta}^+ : B\SO \to B\OO$, for which $\bar{\theta}^+$-structures are orientations, or $\bar{\theta}^{\sfr} : E\OO \to B\OO$, for which $\bar{\theta}^{\sfr}$-structures are stable framings. 

As the map $\tau_M$ is not literally unique---it is only unique up to coherent homotopies---in order to study the action of $\Diff(M)$ on $\bar{\theta}$-structures it is best to take another model. If $M$ is $d$-dimensional then we define a corresponding (unstable) tangential structure via the pullback
\begin{equation}\label{eq:UnstableTS}
\begin{tikzcd}
B_d \rar \dar{\theta} & B \dar{\bar{\theta}}\\
B\OO(d) \rar & B\OO,
\end{tikzcd}
\end{equation}
so that there is a $d$-dimensional vector bundle $\theta^*\gamma_d \to B_d$ classified by $\theta$. Then we define the space of $\theta$-structures on $M$ to be the space of bundle maps
$$\Theta(M) := \mathrm{Bun}(TM, \theta^*\gamma_d),$$
i.e.\ continuous maps $TM \to \theta^*\gamma_d$ which cover a map $M \to B_d$ and are linear isomorphisms on each fibre. This has an evident right action of $\Diff(M)$, by precomposing with the differential of a diffeomorphism. We write
$$\theta(M) := \pi_0 \Theta(M),$$
which inherits a right $\pi_0\Diff(M)$-action.

Choosing a map $\tau_M : M \to B\OO(d)$ classifying the tangent bundle of $M$, it is easy to show (using that $\mathrm{Bun}(TM, \gamma_d) \simeq *$, which is the defining property of the universal bundle $\gamma_d \to B\OO(d)$) that $\Theta(M)$ is homotopy equivalent to the space of lifts in the diagram
\begin{equation*}
\begin{tikzcd}
 & B_d \dar{\theta}\\
M \rar{\tau_M} \arrow[ru, dashed] & B\OO(d),
\end{tikzcd}
\end{equation*}
and so by cartesianness of \eqref{eq:UnstableTS} it is equivalent to the space of lifts of the map $\tau^s_M$ along $\bar{\theta}$, as we originally defined a $\bar{\theta}$-structure to be. By the cartesianness of \eqref{eq:UnstableTS} we therefore blur the distinction between $\bar{\theta}$- and $\theta$-structures.

\begin{remark}
In the discussion above it is clear that it does not matter that $\theta$ is obtained by pulling back from a fibration over $B\OO$, and it is common to refer to a fibration over $B\OO(d)$ as a tangential structure.
\end{remark}

\subsection{The tangential structure $\theta^{\hyp}$}\label{sec:ThetaHyp}

There are two natural tangential structures in the situation we are considering. The tangent bundle of a degree $d$ hypersurface $X \subset \bC\bP^4$ satisfies $TX \oplus \mathcal{O}(d)\vert_X  = T\mathbb{CP}^4\vert_X$ and so using that $T\mathbb{CP}^4 \oplus \underline{\bC} \cong \mathcal{O}(1)^{\oplus 5}$ we have
\begin{equation}\label{eq:TangBundle}
TX \oplus \mathcal{O}(d)\vert_X \oplus \underline{\bC} \cong \mathcal{O}(1)^{\oplus 5}\vert_X.
\end{equation}
If we define maps
$$\bar{\theta}^{\hyp}: \bC\bP^\infty \xrightarrow{\mathcal{O}(1)^{\oplus 5} - \mathcal{O}(d)-\underline{\bC}^4} B\U \overset{\bar{\theta}^\bC}\lra B\OO,$$
which we implicitly convert into Serre fibrations, then the identity \eqref{eq:TangBundle} shows that $X$ has a canonical $\theta^{\hyp}$-structure $\ell_X^{\hyp}$, and so also a canonical $\theta^\bC$-structure $\ell_X^\bC$.

More generally, the universal family of smooth hypersurfaces $\mathcal{X}_d \to \mathcal{U}_d$ comes with a fibrewise embedding $\mathcal{X}_d \subset \mathcal{U}_d \times \mathbb{CP}^4$, and the vertical tangent bundle of this family satisfies $T_v \mathcal{X}_d \oplus \mathcal{O}(d)\vert_{\mathcal{X}_d} \oplus \underline{\bC} \cong \mathcal{O}(1)^{\oplus 5}\vert_{\mathcal{X}_d}$, giving a lift of the map $\tau_v : \mathcal{X}_d \to B\OO$ classifying $T_v \mathcal{X}_d$ along $\bar{\theta}^{\hyp}$, i.e.\ a $\theta^{\hyp}$-structure on the whole family. This implies that $\theta^{\hyp}$-structure $\ell_{X_d}^{\hyp}$ is preserved by $\Mon_d$. Thus the $\theta^\bC$-structure $\ell_{X_d}^\bC$ is also preserved by $\Mon_d$, which is unsurprising as it is even preserved by the symplectic mapping class group $\pi_0 \mathrm{Symp}(X_d)$. Somewhat surprisingly we have the following.

\begin{lemma}\label{lem:ThetaCFixedByMCG}
The $\theta^\bC$-structure $\ell_{X_d}^\bC$ is preserved by $\MCG_d$.
\end{lemma}
\begin{proof}
As $\bar{\theta}^\bC : B\U \to B\OO$ is a fibration of $H$-spaces (in fact of infinite loop spaces) with fibre the $H$-space $\OO/\U$, the set $\theta^\bC(X_d)$ is a torsor for the abelian group $[X_d, O/U]$, with action
$$- \cdot - :  \theta^\bC(X_d) \times [X_d, \OO/\U] \lra \theta^\bC(X_d).$$
An orientation-preserving diffeomorphism $\varphi : X_d \to X_d$ then has $ \ell^\bC_{X_d} \circ D\varphi = \ell^\bC_{X_d} \cdot \delta^\bC(\varphi) $ for a unique $\delta^\bC(\varphi) \in [X_d, \OO/\U]$, and this defines a function
$$\delta^\bC : \MCG_d \lra [X_d, \OO/\U]$$
which is a (right) crossed homomorphism with respect to the right action of $\MCG_d$ on $[X_d, \OO/\U]$ by precomposition.

Using the Atiyah--Hirzebruch spectral sequence to calculate $[X_d, \OO/\U]$, we see it has a filtration with filtration quotients
$$H^0(X_d ; \bZ/2) = \bZ/2, \quad H^2(X_d;\bZ) = \bZ\{x\}, \quad H^6(X_d;\bZ) = \bZ\{xy\}.$$
The projection of $\delta^\bC(\varphi)$ to the first quotient $H^0(X_d ; \bZ/2)$ is always trivial, as this records whether $\varphi$ preserves orientation. Neglecting this first quotient we see that the remainder $[X_d, \SO/\U]$ of $[X_d, \OO/\U]$ is torsion-free, and the Chern--Dold character gives an isomorphism
$$\mathrm{ch} : [X_d, \SO/\U] \otimes \bQ \overset{\sim}\lra H^{4*+2}(X_d, \bQ).$$
In other words $\delta^\bC(\varphi)$ is trivial if and only if $\varphi$ fixes $c_1(X_d)$ and $c_3(X_d)$. But all orientation-preserving diffeomorphisms act trivially on $H^{even}(X_d;\bZ)$, so they all preserve $c_1(X_d)$ and $c_3(X_d)$.
\end{proof}

\begin{remark}\label{rem:MCGPresUnstableACStr}
There is an unstable tangential structure $\theta^{ac}: B\U(3) \to B\OO(6)$, classifying almost complex structures on 6-manifolds: $\bar{\theta}^\bC : B\U \to B\OO$ classifies \emph{stable} almost complex structures. Although we will not need it, we record the fact that $\MCG_d$ also preserves any almost complex structure. % refining $\ell_{X_d}^\bC$. 
The following argument, along lines suggested by the referee, replaces an earlier more complicated one.

First, we choose a lift $\theta^{ac, +} : B\U(3) \to B\SO(6)$: as $\MCG_d$ preserves orientations by definition, we may as well study its action on lifts along this map. There is a homeomorphism $\SO(6)/\U(3) \cong \mathbb{CP}^3$ (arising from $\mathrm{Spin}(6) \cong \mathrm{SU}(4)$, and identifying the preimage of $\U(3) \leq \SO(6)$ in $\mathrm{Spin}(6)$ with $\mathrm{S}(\U(1) \times \U(3)) \leq \mathrm{SU}(4)$). The maps classifying the the first Chern class and second Stiefel--Whitney class give the bottom homotopy commutative square in the diagram
\begin{equation*}
\begin{tikzcd}
\mathbb{CP}^3 \arrow[r, dashed] \dar & K(\bZ,2) \dar{2} \\
B\U(3) \rar{c_1} \dar{\theta^{ac, +}} & K(\bZ, 2) \dar{\rho}\\
B\SO(6) \rar{w_2} & K(\bZ/2, 2),
\end{tikzcd}
\end{equation*}
and choosing a homotopy making this square commute yields the top dashed map. By considering the induced map of Serre spectral sequences we see that the top map is an isomorphism on $H_2(-;\bZ)$ and so on $\pi_2(-)$, and so is 7-connected. It follows that the map from $\Theta^{ac, +}(X_d)$ to the space of lifts
\begin{equation*}
\begin{tikzcd}
 & & K(\bZ,2) \dar{\rho}\\
X_d \arrow[rr, "w_2(X_d)"] \arrow[rru, dashed] & & K(\bZ/2,2)
\end{tikzcd}
\end{equation*}
is $(7-\dim(X_d) = 1)$-connected. In particular $\theta^{ac, +}(X_d) = \pi_0\Theta^{ac, +}(X_d)$ is canonically identified with the set of lifts of $w_2(X_d) \in H^2(X_d;\bZ/2)$ to an integral cohomology class. As $\MCG_d$ acts trivially on the second integral cohomology of $X_d$, it acts trivially on this set of lifts.
\end{remark}

It is not the case that $\MCG_d$ preserves $\ell^{\hyp}_{X_d}$, and our next goal is to determine $\Stab_{\MCG_d}(\ell^{\hyp}_{X_d})$. The methods we will use are related to those of \cite{KR-WFram, KrannichMCG}. We begin by an analysis similar to the proof of Lemma \ref{lem:ThetaCFixedByMCG}. The map $\bar{\theta}^{\hyp} : \mathbb{CP}^\infty \to B\OO$ is no longer a map of $H$-spaces, but we can proceed as follows. We have explained that $\Theta^{\hyp}(X_d)$ is homotopy equivalent to the space of lifts of $\tau^s_{X_d} : X_d \to B\OO$ along $\bar{\theta}^{\hyp}$, which shows that it fits into a homotopy cartesian square
\begin{equation*}
\begin{tikzcd}
\Theta^{\hyp}(X_d) \rar \dar & \map(X_d, \bC\bP^\infty) \dar{(\bar{\theta}^{\hyp})_*}\\
\{\tau_{X_d}\} \rar& \map(X_d, B\OO).
\end{tikzcd}
\end{equation*}
On homotopy groups this gives a sequence
$$0 = H^1(X_d ; \bZ) \lra KO^{-1}(X_d) \overset{\curvearrowright}\lra \theta^{\hyp}(X_d) \overset{\hat{x}}\lra H^2(X_d ; \bZ) \overset{(\bar{\theta}^{\hyp})_*}\lra KO^{0}(X_d)$$
which is exact exact in the sense of groups and pointed sets: the two leftmost terms are groups, the map marked $\curvearrowright$ is an action, and the orbits of this action are in bijection with $((\bar{\theta}^{\hyp})_*)^{-1}(0)$. The map $\hat{x}$ assigns to a lift $\ell^{\hyp} : X_d \to \bC\bP^\infty$ the class $(\ell^{\hyp})^*(\iota_2) \in H^2(X_d;\bZ)$. 

As an orientation-preserving diffeomorphism $\varphi : X_d \to X_d$ acts trivially on $H^2(X_d;\bZ)$ it follows that $\ell^{\hyp}_{X_d} \circ D\varphi$ and $\ell^{\hyp}_{X_d}$ lie in the same fibre of $\hat{x}$: thus there is a unique $\delta^{\hyp}(\varphi) \in KO^{-1}(X_d)$ such that $\ell^{\hyp}_{X_d} \circ D\varphi = \ell^{\hyp}_{X_d} \cdot \delta^{\hyp}(\varphi)$. As in the proof of Lemma \ref{lem:ThetaCFixedByMCG} this defines a function
$$\delta^{\hyp} : \MCG_d \lra KO^{-1}(X_d)$$
which is a crossed homomorphism with respect to the right action of $\MCG_d$ on $KO^{-1}(X_d)$ by contravariant functoriality. Using the Atiyah--Hirzebruch spectral sequence to calculate $KO^{-1}(X_d)$, we see it has a filtration with filtration quotients
$$H^0(X_d ; \pi_0(\OO)) = \bZ^\times, \quad H^3(X_d;\pi_3(\OO)).$$
As in the lemma, projection of $\delta^{\hyp}(\varphi)$ to the first quotient detects whether $\varphi$ preserves orientation, which it does, so $\delta^{\hyp}$ refines to a crossed homomorphism
$$\delta^{\hyp} : \MCG_d \lra H^3(X_d ; \pi_3(\OO)) \subset KO^{-1}(X_d).$$
We then have
$$\Stab_{\MCG_d}(\ell^{\hyp}_{X_d}) = \mathrm{Ker}(\delta^\hyp : \MCG_d \to H^3(X_d ; \pi_3(\OO))).$$

\begin{remark}
The stable tangential structure $\bar{\theta}^\hyp : \mathbb{CP}^\infty \to B\OO$ is not the most sophisticated structure with which we can endow a hypersurface $X \subset \mathbb{CP}^4$. The equation $TX \oplus \mathcal{O}(d)\vert_X = T\mathbb{CP}^4\vert_X$ shows that it may be endowed with the (unstable) tangential structure given by the homotopy equaliser
\begin{equation*}
\begin{tikzcd}
B \rar & B\U(3) \times \mathbb{CP}^4 \arrow[r,shift left=.75ex] \arrow[r,shift right=.75ex] & B\U(4)
\end{tikzcd}
\end{equation*}
of the maps classifying $\gamma_3^\bC \oplus \mathcal{O}(d)$ and $T\mathbb{CP}^4$ respectively, made into a tangential structure via the natural maps $B \to B\U(3) \times \mathbb{CP}^4 \to B\U(3) \to B\OO(6)$. It is more complicated to analyse the set of such structures on $X_d$, and it seems likely that for reasons similar to Remark \ref{rem:MCGPresUnstableACStr} there is no advantage in doing so.
\end{remark}

\subsection{Distortion of the first Pontrjagin class}\label{sec:Distorsion}

We wish to relate $\delta^\hyp$ to Kreck and Su's map $v_p : \mathrm{SMCG}_d \to H^3(X_d;\bZ)$ from \eqref{eq:KS}, which is described in terms of Sullivan's ``distortion of the first Pontrjagin class'', see \cite[\S 13]{sullivaninf}. As explained by Hain \cite[\S 5.2]{HainMCGKahler}, in the case at hand this may be implemented as follows.

Let $\varphi \in \MCG_d$ act trivially on cohomology, i.e.\ lie in the \emph{Torelli subgroup} $\mathrm{TMCG}_d := \mathrm{Ker}(\MCG_d \to \mathrm{Aut}(H^3(X_d;\bZ)))$, and let $T_\varphi$ denote its mapping torus. The long exact sequence on cohomology for the pair $(T_\varphi, X_d)$, using excision, takes the form
$$\cdots \lra H^3(X_d;\bZ) \overset{\mathrm{Id} - \varphi^* = 0}\lra H^3(X_d;\bZ) \overset{\partial}\lra H^4(T_\varphi;\bZ) \lra H^4(X_d;\bZ) \lra \cdots.$$
There is a unique $\bar{x} \in H^2(T_\varphi;\bZ)$ which restricts to $x \in H^2(X_d;\bZ)$. On $X_d$ we have $p_1(TX_d) = (5-d^2)x^2$, so the class $p_1(T(T_\varphi)) - (5-d^2)\bar{x}^2 \in H^4(T_\varphi;\bZ)$ vanishes on $X_d$ and so comes from a unique class $\Delta_{p_1}(\varphi) \in H^3(X_d;\bZ)$. This defines a function
$$\Delta_{p_1} : \mathrm{TMCG}_d \lra H^3(X_d;\bZ),$$
which is easily checked to be a homomorphism.

\begin{lemma}\label{lem:Dist1}
We have 
$$\Delta_{p_1} : \mathrm{TMCG}_d \overset{\delta^\hyp\vert_{\mathrm{TMCG}_d}}\lra H^3(X_d ; \pi_3(\OO)) \overset{(p_1)_*}\lra H^3(X_d ; \bZ).$$
Under the usual identification $\pi_3(\OO) = \bZ$, the latter map is multiplication by 2.
\end{lemma}
\begin{proof}
For $\varphi$ in the Torelli group, we describe $\delta^\hyp(\varphi) \in KO^{-1}(X_d)$ in a similar way to $\Delta_{p_1}$. The long exact sequence on $KO$-theory for the pair $(T_\varphi, X_d)$, using excision, takes the form
$$\cdots \lra KO^{-1}(X_d) \overset{Id - \varphi^*}\lra KO^{-1}(X_d) \overset{\partial}\lra KO^0(T_\varphi) \lra KO^0(X_d) \lra \cdots.$$
The short exact sequence $0 \to H^3(X_d;\bZ) \to KO^{-1}(X_d) \to H^0(X_d;\bZ/2) \to 0$ is functorially split by $* \to X_d \to *$, so the fact that $\varphi$ acts as the identity on cohomology shows that the left-hand map in the above sequence is zero. Using the map $\bar{x} : T_\varphi \to K(\bZ,2) = \mathbb{CP}^\infty$ we can form the class
$$(T(T_\varphi) - \epsilon^{7}) - \bar{x}^*\bar{\theta}^{\hyp} \in KO^0(T_\varphi).$$
This vanishes on $X_d$, and so comes from a unique class in $KO^{-1}(X_d)$: this is $\delta^{\hyp}(\varphi)$.

The first Pontrjagin class gives a homomorphism $p_1 : KO^0(T_\varphi) \to H^4(T_\varphi;\bZ)$, which extends to a map between the two exact sequences above. The composition $H^3(X_d;\pi_3(\OO)) \subset KO^{-1}(X_d) \overset{p_1}\to H^3(X_d;\bZ)$ is the map induced by $p_1 : \pi_3(\OO) \to \bZ$, which is well-known to be an isomorphism onto the subgroup $2 \bZ$.
\end{proof}

We now relate this to the work of Kreck and Su, namely to the map $v_p$ in the central extension
$$0 \lra \mathrm{K}_d \lra \mathrm{SMCG}_d \overset{v_p}\lra H^3(X_d;\bZ) \lra 0,$$
where $\mathrm{K}_d$ is a certain finite abelian group, and $\mathrm{SMCG}_d \leq \MCG_d$ is the subgroup of those diffeomorphisms which act trivially on $\pi_3(X_d)$. This map is made explicit by the diagram before the statement of \cite[Theorem 2.6]{KSv3} and the definition of the map $v_{x,p}$ just before \cite[Proposition 4.5]{KSv3}. Together those show that:

\begin{lemma}\label{lem:Dist2}
We have  $\Delta_{p_1}\vert_{\mathrm{SMCG}_d} = 4 \cdot v_p$.\qed
\end{lemma}

Restricted to $\mathrm{SMCG}_d$ the crossed homomorphism $\delta^{\hyp} : \MCG_d \to H^3(X_d ; \bZ)$ is a homomorphism, so it annihilates the finite group $\mathrm{K}_d$ and hence descends to a homomorphism
$$\delta^{\hyp} : \mathrm{SMCG}_d / \mathrm{K}_d  \lra H^3(X_d ; \pi_3(\OO)).$$
\begin{corollary}\label{cor:tIsMultBy2}
This is an isomorphism onto $2\cdot H^3(X_d ; \pi_3(\OO))$.
\end{corollary}
\begin{proof}
Postcomposing with the map $(p_1)_* : H^3(X_d ; \pi_3(\OO)) \to H^3(X_d ; \bZ)$, which is an isomorphism onto $2 \cdot H^3(X_d ; \bZ)$, gives the map $\Delta_{p_1}$ by Lemma \ref{lem:Dist1}, which is the map $4 \cdot v_p$ by Lemma \ref{lem:Dist2}, so is an isomorphism onto $4 \cdot H^3(X_d ; \bZ)$ as $v_p$ is surjective.
\end{proof}

In particular $\Stab_{\MCG_d}(\ell^{\hyp}_{X_d}) = \mathrm{Ker}(\delta^\hyp : \MCG_d \to H^3(X_d;\pi_3(\OO)))$ intersects $\mathrm{SMCG}_d$ precisely in $\mathrm{K}_d$, giving a half-exact sequence
\begin{equation}\label{eq:StabExtHalf}
1 \lra \mathrm{K}_d \lra \Stab_{\MCG_d}(\ell^{\hyp}_{X_d}) \lra \mathrm{Aut}(\pi_3(X_d), \lambda).
\end{equation}

\subsection{The surgery kernel}\label{sec:SurgeryKernel}

The action of $\MCG_d$ on the set $\theta^{\hyp}(X_d)$ preserves the subset $\theta^{\hyp}(X_d; x, +)$ of those $\theta^\hyp$-structures which map to $x \in H^2(X_d;\bZ)$ under $\hat{x}$ and which induce the standard orientation of $X_d$. This action descends to an action of $\mathrm{Aut}(\pi_3(X_d), \lambda)$ on the set 
$$\theta^{\hyp}(X_d; x, +)/\mathrm{SMCG}_d,$$
and the discussion of the previous section shows that this set is a torsor for $H^3(X_d ; \bZ/2) = H^3(X_d;\pi_3(\OO))/\mathrm{Im}(\delta^{\hyp}\vert_{\mathrm{SMCG}_d})$. We wish to identify this set with something more meaningful. Associated to a $\theta^{\hyp}$-structure $\ell$ there is a quadratic form $\mu_\ell : \pi_3(X_d) \to \bZ/2$ on
$$\pi_3(X_d) \overset{\sim}\longleftarrow \pi_4(\bC\bP^\infty, X_d) \overset{\sim}\lra H_4(\bC\bP^\infty, X_d;\bZ)$$
obtained by considering this as the surgery kernel for the normal map $\ell : X_d \to \bC\bP^\infty$ covered by the corresponding bundle map, see \cite[p.\ 728]{Kreck}. This $\mu_\ell$ is a quadratic refinement of the intersection form $\lambda$ in the sense that
$$\mu_\ell(a+b) = \mu_\ell(a) + \mu_\ell(b) + \lambda(a,b) \mod 2.$$
In terms of the extension \eqref{eq:pi3defn}, as the subgroup $\bZ/d\{\eta\} \leq \pi_3(X_d)$ is radical with respect to $\lambda$, the restriction $\mu_\ell\vert_{\bZ/d} : \bZ/d \to \bZ/2$ is a homomorphism. In particular if $d$ is odd then this homomorphism must be trivial, and so $\mu_\ell$ descends to a quadratic refinement of the intersection form on $H_3(X_d;\bZ)$.

\begin{remark}\label{rem:QuadRefExists}
The quadratic refinement $\mu_{\ell^{\hyp}_{X_d}}$ is $\Mon_d$-invariant, and when $d$ is odd the induced quadratic refinement on $H_3(X_d;\bZ)$ supplies the missing ingredient in Beauville's Theorem \ref{thm:Beauville}. The argument of \cite[Section 2]{Wood} applies to this quadratic refinement, showing that it has Arf invariant $0$ if $d \equiv \pm 1 \mod 8$ and Arf invariant $1$ if $d \equiv \pm 3 \mod 8$.
\end{remark}

If $d$ is even then this is not the case:

\begin{lemma}\label{lem:QuadDescends}
If $d$ is even then the homomorphism $\mu_\ell\vert_{\bZ/d} : \bZ/d \to \bZ/2$ is surjective.
\end{lemma}
\begin{proof}
If not then $\mu_\ell$ would descend to a quadratic refinement of the intersection form on $H_3(X_d;\bZ)$, which would be $\Mon_d$-invariant, contradicting Beauville's Theorem \ref{thm:Beauville} in the case $d$ even.
\end{proof}

Let us write $\mathrm{Quad}(\pi_3(X_d), \lambda)$ for the set of quadratic refinements $\mu$ of $(\pi_3(X_d), \lambda)$ whose restriction to $\bZ/d \leq \pi_3(X_d)$ is zero if $d$ is odd, and non-zero if $d$ is even. If $\mu$ and $\mu'$ are such quadratic refinements then $\mu-\mu' : \pi_3(X_d) \to \bZ/2$ vanishes on $\bZ/d \leq \pi_3(X_d)$ so descends to a homomorphism $H_3(X_d;\bZ) \to \bZ/2$, and hence $\mathrm{Quad}(\pi_3(X_d), \lambda)$ forms a $H^3(X_d;\bZ/2)$-torsor. 

The construction above defines a function 
\begin{align*}
\Phi: \theta^{\hyp}(X_d;x, +) &\lra \mathrm{Quad}(\pi_3(X_d), \lambda)\\
\ell &\longmapsto \mu_\ell.
\end{align*}
This is equivariant when the right-hand side is equipped with the right $\mathrm{Aut}(\pi_3(X_d), \lambda)$-action by precomposition, so it descends to a $\mathrm{Aut}(\pi_3(X_d), \lambda)$-equivariant function
$$\Phi : \theta^{\hyp}(X_d;x,+)/\mathrm{SMCG}_d \lra \mathrm{Quad}(\pi_3(X_d), \lambda).$$
This is easily checked to be a map of $H^3(X_d;\bZ/2)$-torsors, so is a bijection.

\begin{corollary}
The image of $\Stab_{\MCG_d}(\ell^{\hyp}_{X_d})$ in $\mathrm{Aut}(\pi_3(X_d), \lambda)$ is the stabiliser $\mathrm{Aut}(\pi_3(X_d), \lambda, \mu)$ of the quadratic form $\mu := \mu_{\ell^{\hyp}_{X_d}}$.\qed
\end{corollary}

We can therefore improve the half-exact sequence \eqref{eq:StabExtHalf} to an exact sequence
\begin{equation}\label{eq:StabExt}
1 \lra \mathrm{K}_d \lra \Stab_{\MCG_d}(\ell^{\hyp}_{X_d}) \lra \mathrm{Aut}(\pi_3(X_d), \lambda, \mu) \lra 1.
\end{equation}

%\begin{comment}
\begin{remark}\label{rem:Brown}
The referee suggests the following perspective. The (stable) tangential structure $\bar{\theta}^\hyp : \mathbb{CP}^\infty \to B\U \to B\OO$, classifying the virtual bundle $\mathcal{O}(1)^{\oplus 5} - \mathcal{O}(d) - \underline{\bC}^4$, has total Stiefel--Whitney class 
$$w(\bar{\theta}^\hyp) = \frac{(1+x)^5}{1+dx} = 1 + (1+d) x + x^4 + \cdots \in H^*(\mathbb{CP}^\infty ; \bZ/2) = \bZ/2[[x]]$$
and so its 4th Wu class $v_4 = w_4 + w_3 w_1 + w_2^2 + w_1^4$ is given by $v_4(\bar{\theta}^\hyp) = (1+d)  x^2$. This vanishes precisely when $d$ is odd.

In this case, the work of Brown \cite{Brown} assigns to any closed 6-manifold $M$ endowed with a $\bar{\theta}^\hyp$-structure $\ell$ a quadratic form $q_{\ell}: H^3(M;\bZ/2) \to \bZ/2$. To see this, in the notation of Brown's paper one calculates that 
$$\{S^6, \mathrm{Th}((\bar{\theta}^\hyp)^*\gamma \to \mathbb{CP}^\infty) \wedge K_3\} = \begin{cases}\bZ/2\{\lambda\} & d \text{ odd}\\
0 & d \text{ even},
\end{cases}$$
so when $d$ is odd this has a unique map $h$ to $\bZ/4$ with $h(\lambda)=2$. By the procedure on \cite[p.\ 370]{Brown} this yields a quadratic form $q_{\ell} = \varphi_h : H^3(M;\bZ/2) \to \bZ/4$ taking values in $2\bZ/4 \cong \bZ/2$. For $(X_d, \ell_{X_d}^\hyp)$ it should be possible to equate $q_{\ell_{X_d}^\hyp}$ (composed with reduction modulo 2) with our $\mu_{\ell_{X_d}^\hyp}$ by a careful application of \cite[Corollary (1.13)]{Brown} but we offer the following indirect argument: both quadratic refinements are constructed using only the $\bar{\theta}^\hyp$-structure so are invariant under the action of $\Mon_d$, but by Beauville's theorem \ref{thm:Beauville} there is only one $\Mon_d$-invariant quadratic refinement, so they must be equal.
\end{remark}
%\end{comment}

\section{Lower bounds: automorphisms of the Fermat hypersurface}\label{sec:LowerBounds}

The goal of this section is to show that the image of
$$\Mon_d \lra \Stab_{\MCG_d}(\ell^{\hyp}_{X_d})$$
is not too small, by showing that it contains $\mathrm{Im}(\Phi : \Theta_7 \to \mathrm{K}_d)$, and
 surjects onto $\mathrm{Aut}(\pi_3(X_d), \lambda, \mu)$. The first of these is an immediate consequence of the work of Krylov \cite{Krylov}, and is very general: it just uses that $X_d$ is a 3-fold which admits a deformation to an $A_2$-singularity. The latter is more specific, and proceeds by analysing the effect of automorphisms of the Fermat hypersurface on $\pi_3(X_d)$, following Pham \cite{Pham} and Looijenga \cite{LooijengaFermat}.

\subsection{Dehn twists and the Milnor sphere}

The following simple consequence of the work of Krylov is rather surprising. It presumably admits generalisations to higher dimensions, and can perhaps be thought of as the analogue of the construction of $bP$-spheres as Brieskorn varieties, but for automorphisms.

\begin{theorem}\label{thm:Krylov}
For $d \geq 3$ the subgroup $\mathrm{Im}(\Phi : \Theta_7 \to \mathrm{K}_d) \leq \MCG_d$ is contained in the image of $\alpha: \Mon_d \to \MCG_d$.
\end{theorem}

\begin{proof}
Under the given condition $X_d$ admits a deformation to an $A_2$-singularity, meaning that there is an orientation-preserving embedding $T^* S^3 \natural T^* S^3 \subset X_d$ of the plumbing such that the Dehn twists around each $S^3$ lies in the image of $\alpha$. Now $T^* S^3 \natural T^* S^3 \cong (S^3 \times S^3) \setminus D^6 \subset S^3 \times S^3$ and, considering $S^3 \subset \mathbb{H}$ as the unit quaternions, inside the latter manifold the two Dehn twists are isotopic to
$$Y : (u,v) \mapsto (u, uv) \quad\text{ and }\quad U : (u,v) \mapsto (v^{-1} u, v).$$
By \cite[Theorem 3]{Krylov} these satisfy 
$$(YUY)^4 = \Sigma_{\text{Milnor}}^{-1} \in \Theta_7 \leq \pi_0(\Diff^+(S^3 \times S^3))$$
where $ \Sigma_{\text{Milnor}}^{-1}$ denotes the inverse (in the group $\Theta_7$) of the class of the Milnor sphere. On the other hand it follows from \cite[p.\ 657 and Lemma 3 b)]{kreckisotopy} that the map
$$\pi_0(\Diff_c(T^* S^3 \natural T^* S^3)) \cong \pi_0(\Diff_{c}((S^3 \times S^3) \setminus D^6)) \lra \pi_0(\Diff^+(S^3 \times S^3)),$$
which extends by the identity over $D^6$, is an isomorphism. Thus the given expression represents $\Sigma_{\text{Milnor}}^{-1} \in \Theta_7$ in the plumbing, and so when put inside $X_d$ also represents the generator $[\Sigma_{\text{Milnor}}^{-1}] \in \mathrm{Im}(\Theta_7 \to \mathrm{K}_d)$.
\end{proof}

\subsection{Automorphisms of $\pi_3(X_d)$ realised by monodromy}\label{sec:AutMonodromy}

By the discussion in Section \ref{sec:SurgeryKernel} the quadratic refinement $\mu$ on $\pi_3(X_d)$ descends to a quadratic refinement on $H_3(X_d;\bZ)$ if and only if $d$ is odd. The map
$$\rho : \mathrm{Aut}(\pi_3(X_d), \lambda, \mu) \lra \begin{cases}
\mathrm{Aut}(H_3(X_d;\bZ), \lambda, \mu) & d \text{ odd}\\
\mathrm{Aut}(H_3(X_d;\bZ), \lambda) & d \text{ even}
\end{cases}$$
is surjective, and its kernel is identified with the group of automorphisms of the extension
\begin{equation*}
0 \lra \bZ/d\{\eta\} \lra \pi_3(X_d) \lra H_3(X_d;\bZ) \lra 0
\end{equation*}
which are the identity on the outer terms, and preserve $\mu$. There is an isomorphism
$$\mathrm{Ker}(\rho) \overset{\sim}\lra \mathrm{Hom}\left(H_3(X_d;\bZ), \begin{cases}
\bZ/d & d \text{ odd}\\
\mathrm{ker}(\mu : \bZ/d \to \bZ/2) & d \text{ even}
\end{cases}\right)$$
given as follows: for $\phi \in \mathrm{Ker}(\rho)$ and $z \in H_3(X_d;\bZ)$, choose a lift $\bar{z} \in \pi_3(X_d)$ and form $\phi(\bar{z})-\bar{z} \in \mathrm{Ker}(\pi_3(X_d) \to H_3(X_d;\bZ)) = \bZ/d\{\eta\}$, which does not depend on the choice of lift $\bar{z}$ of $z$. Then $\mu(\phi(\bar{z})-\bar{z}) = \mu(\phi(\bar{z})) + \mu(\bar{z}) - \lambda(\phi(\bar{z}), \bar{z}) = 0$ as the first two terms cancel (because $\phi$ preserves $\mu$) and the latter is $\lambda(z,z)=0$ (because $\lambda$ factors over $H_3(X_d;\bZ)$ and is antisymmetric). This map is easily checked to be an isomorphism. To avoid distinguishing cases, we can write the target above as $H^3(X_d; 2 \cdot \bZ/d)$.

\begin{lemma}\label{lem:ConstrainedSubgroups}
If $G \leq \mathrm{Aut}(\pi_3(X_d), \lambda, \mu)$ is a subgroup which satisfies $\rho(G) = \rho(\mathrm{Aut}(\pi_3(X_d), \lambda, \mu))$, then $\mathrm{Ker}(\rho\vert_G) = k \cdot H^3(X_d; 2 \cdot \bZ/d)$ for some $k \in \bZ$, under the identification given above.
\end{lemma}
\begin{proof}
Under the surjection $1 \mapsto 2 : \bZ \to 2 \cdot \bZ/d$, the preimage of $\mathrm{Ker}(\rho\vert_G)$ is a subgroup of $H^3(X_d;\bZ)$, and is preserved by the action of $\mathrm{Aut}(\pi_3(X_d), \lambda, \mu)$ on this group via $\rho$. It therefore suffices to show that all subgroups $I \leq H^3(X_d;\bZ)$ which are preserved by $\mathrm{Aut}(H_3(X_d;\bZ), \lambda, q)$ and some quadratic refinement $q$ are of the form $k \cdot H^3(X_d;\bZ)$.

The group $\mathrm{Aut}(H_3(X_d;\bZ), \lambda, q)$ acts transitively on the set of unimodular elements of the same $q$-length (this is not hard to show by hand, but it follows from \cite[Corollary 3.13]{Friedrich} using that $\mathrm{usr}(\bZ) = 2$ and $d \geq 3$ so that $(H_3(X_d;\bZ), \lambda)$ contains $\geq 4$ hyperbolic forms and hence $(H_3(X_d;\bZ), \lambda, q)$ contains $\geq 3$ hyperbolic forms). Choose a symplectic basis $\{e_1, f_1, \ldots, e_g, f_g\}$ for $(H_3(X_d;\bZ), \lambda)$, such that $q(e_i)=q(f_i)=0$ for $i>1$, and $q(e_1)=q(f_1) = \mathrm{Arf}(q)$.

Let $k := \min\{|\lambda(a,b)| \neq 0 \, : \, a \in I, b \in H^3(X_d;\bZ)\}$, then all elements of $I$ are divisible by $k$ in $H^3(X_d;\bZ)$, so $I \leq k \cdot H^3(X_d;\bZ)$. Choose an $a_0 \in I$ such that $\lambda(a_0, b_0)=k$. Then $a_0 = k \cdot a_0'$ for a unimodular $a'_0 \in H^3(X_d;\bZ)$. 

If $q(a_0')=0$ then there are automorphisms sending $a_0'$ to: (i) $e_i$ or $f_i$ for any $i>1$, (ii) $e_1 + \mathrm{Arf}(q)(e_2+f_2)$ or $f_1 + \mathrm{Arf}(q)(e_2+f_2)$. Thus $I$ contains $k \cdot e_i$ and $k \cdot f_i$ for all $i$, so $k \cdot H^3(X_d;\bZ) = I$.

If $q(a_0')=1$ then there are automorphisms sending $a_0'$ to: (i) $e_1 + f_1$, (ii) $e_1+f_1 + e_i$ or $e_1+f_1 + f_i$ for any $i\geq 1$, (iii) $e_1 + (1-\mathrm{Arf}(q))(e_2+f_2)$ or $f_1 + (1-\mathrm{Arf}(q))(e_2+f_2)$. Thus $I$ contains $k \cdot e_i$ and $k \cdot f_i$ for all $i$, so $k \cdot H^3(X_d;\bZ) = I$.
\end{proof}

\begin{proposition}\label{prop:EpiCriterion}
If $G \leq \mathrm{Aut}(\pi_3(X_d), \lambda, \mu)$ is a subgroup which satisfies $\rho(G) = \rho(\mathrm{Aut}(\pi_3(X_d), \lambda, \mu))$, then 
$$H_0(G ; \pi_3(X_d)) \cong \bZ/(k,d)\{\eta\},$$
for $k$ as in Lemma \ref{lem:ConstrainedSubgroups}. Thus $G = \mathrm{Aut}(\pi_3(X_d), \lambda, \mu)$ if and only if the composition
$$\bZ/d\{\eta\} \lra \pi_3(X_d) \lra H_0(G ; \pi_3(X_d))$$
is zero.
\end{proposition}
\begin{proof}
Let us abbreviate $A := \mathrm{Aut}(\pi_3(X_d), \lambda, \mu)$, and $H_3(X_d) = H_3(X_d;\bZ)$. Then
\begin{align*}
H_0(A ; \pi_3(X_d)) &= H_0(\mathrm{Im}(\rho) ; H_0(\mathrm{Ker}(\rho) ; \pi_3(X_d)))\\
H_0(G ; \pi_3(X_d)) &= H_0(\mathrm{Im}(\rho) ; H_0(\mathrm{Ker}(\rho\vert_G) ; \pi_3(X_d))).
\end{align*}
and so first consider the long exact sequences for the action of $\mathrm{Ker}(\rho)$ and $\mathrm{Ker}(\rho\vert_G)$ on the extension describing $\pi_3(X_d)$:
\begin{equation*}
\begin{tikzcd}[column sep=0.3cm]
H_1(\mathrm{Ker}(\rho\vert_G) ; H_3(X_d)) \rar{\partial} \dar & \bZ/d \arrow[d, equals] \rar & H_0(\mathrm{Ker}(\rho\vert_G) ; \pi_3(X_d)) \rar \dar & H_0(\mathrm{Ker}(\rho\vert_G) ; H_3(X_d)) \arrow[d, equals]\\
 H_1(\mathrm{Ker}(\rho) ; H_3(X_d)) \rar{\partial} & \bZ/d \rar & H_0(\mathrm{Ker}(\rho) ; \pi_3(X_d)) \rar & H_0(\mathrm{Ker}(\rho) ; H_3(X_d)).
\end{tikzcd}
\end{equation*}
Under the identification $H_1(\mathrm{Ker}(\rho) ; H_3(X_d)) \cong H^3(X_d ; 2\cdot \bZ/d) \otimes H_3(X_d)$ the lower map $\partial$ is given by evaluation, so it has image $2 \cdot \bZ/d$. Similarly the upper map $\partial$ has image $2k \cdot \bZ/d$. This gives a map of short exact sequences, and the induced map on long exact sequence for $\mathrm{Im}(\rho)$-homology takes the form:
\begin{equation*}
\begin{tikzcd}[column sep=0.4cm]
H_1(\mathrm{Im}(\rho) ; H_3(X_d))\rar{\partial} \arrow[d, equals] & \bZ/(2k,d) \rar \arrow[d, two heads] & H_0(G ; \pi_3(X_d)) \rar \dar & H_0(\mathrm{Im}(\rho) ; H_3(X_d)) \rar \arrow[d, equals] & 0\\
H_1(\mathrm{Im}(\rho) ; H_3(X_d))\rar{\partial} & \bZ/(2,d) \rar & H_0(A ; \pi_3(X_d)) \rar & H_0(\mathrm{Im}(\rho) ; H_3(X_d)) \rar & 0
\end{tikzcd}
\end{equation*}
By \cite[Lemma A.2]{KrannichMCG} the rightmost terms vanish, as $\mathrm{Im}(\rho)$ is a (quadratic) symplectic group. 

If $d$ is odd then $\bZ/(2,d)$ vanishes, and if $d$ is even then it is $\bZ/2$. When $d$ is even we claim that the lower $\partial$ is onto. If not, we obtain an $A$-invariant homomorphism $Q : \pi_3(X_d) \to \bZ/2$ which agrees with $\mu$ on $\bZ/d$. Then $\mu - Q$ is again a quadratic refinement of $\lambda$ on $\pi_3(X_d)$, but it vanishes on $\bZ/d \leq \pi_3(X_d)$ and so descends to an $A$-invariant quadratic refinement of $\lambda$ on $H_3(X_d;\bZ)$. By Beauville's Theorem \ref{thm:Beauville} there is no such quadratic refinement. For $d$ both even and odd it then follows that $H_0(A ; \pi_3(X_d))=0$.

The group $H_1(\mathrm{Im}(\rho) ; H_3(X_d))$ has exponent 2 by the centre-kills trick\footnote{The group $\mathrm{Im}(\rho)$ contains the matrix $-\mathrm{Id}$ in its centre, which acts on the module $H_3(X_d)$ by $-1$; conjugation by this matrix induces a map on homology which is both (i) the identity, as it is inner, and (ii) multiplication by $-1$.}, so the image of the top $\partial$ lands in the subgroup of $\bZ/(2k,d)$ of elements of order 2, i.e.\ the subgroup generated by $[k]$. If $d$ is odd this means it vanishes. If $d$ is even it must hit $[k]$, by commutativity of the left-hand square.\footnote{This shows that $\bZ/(2k,d) \to \bZ/(2,d)$ is split, i.e. that $\bZ/(k, d/2)$ has odd order. This is a bit surprising, but is just saying that the subgroups $\mathrm{Ker}(\rho\vert_G) \leq H^3(X_d;2\cdot \bZ/d)$ are even more constrained than Lemma \ref{lem:ConstrainedSubgroups} indicates.} It follows that $H_0(G ; \pi_3(X_d)) \cong \bZ/(k,d)$.
\end{proof}

\begin{theorem}\label{thm:MonodromyOnPi3}
The composition
$$\Mon_d \lra \MCG_d \lra \mathrm{Aut}(\pi_3(X_d), \lambda, \mu)$$
is surjective.
\end{theorem}
\begin{proof}
By Beauville's Theorem \ref{thm:Beauville}, the image $G$ of this composition satisfies $\rho(G) = \rho(\mathrm{Aut}(\pi_3(X_d), \lambda, \mu))$, and so by Proposition \ref{prop:EpiCriterion} the composition is surjective if and only if the composition $\bZ/d\{\eta\} \to \pi_3(X_d) \to H_0(\Mon_d ; \pi_3(X_d))$ is zero.

Let us write $X_d = X_d^3 \subset \bC\bP^4$: the decomposition $\bC\bP^3 \subset \bC\bP^4 \supset \bA^4$ induces a corresponding decomposition $X_d^2 \subset X_d^3 \supset A_d^3$, where $A_d^3$ is the corresponding affine Fermat hypersurface, and $X_d^2 \subset \bC\bP^3$ is the 2-dimensional Fermat hypersurface. The normal bundle of $X_d^2 \subset X_d^3$ is $\mathcal{O}(1)$. The long exact sequence on homology for $(X_d^3, A_d^3)$, along with excision and the Thom isomorphism, gives the bottom part of the following commutative diagram:
\begin{equation*}
\begin{tikzcd}[column sep=0.5cm]
 & & & \bZ/d\{\eta\} \dar\\
 & & \pi_3(A_d^3) \dar{\simeq} \rar & \pi_3(X_d^3) \arrow[d, two heads]\\
H_4(X_d^3;\bZ) \rar \arrow[d, equals] & H_4(X_d^3, A_d^3;\bZ) \rar{\partial} \dar{\simeq} & H_3(A_d^3;\bZ) \rar & H_3(X_d^3;\bZ) \rar & 0\\
\bZ \arrow[rd, "d \cdot -"] & H_2(X_d^2 ; \bZ) \dar{h}\\
  & \bZ
\end{tikzcd}
\end{equation*}
The top part comes from the fact that $A_d^3$ is homotopy equivalent to a wedge of $3$-spheres, and the defining extension for $\pi_3(X_d)$.

\vspace{1ex}

\noindent\textbf{Claim:} The map $\pi_3(A_d^3) \to \pi_3(X_d^3)$ is surjective.

\begin{proof}[Proof of Claim]
We just have to show that it hits $\eta$. If $f : S^3 \to X_d^3$ is a homotopy class which is homologically trivial, then it is some multiple of $\eta$. This multiple may be found as follows: choose a 4-chain bounding $f$ in $X_d^3$, and choose a 4-disc bounding $f$ in $\bC\bP^4$, then intersect the union of these with $\bC\bP^2 \subset \bC\bP^4$ (equivalently, evaluate $h^2$ on this cycle). The result is well-defined only modulo $d$ because $h^2$ evaluated on a 4-cycle in $X_d^3$ lies in $d\bZ$.

By the Lefschetz hyperplane theorem the inclusion $X_d^2 \to \bC\bP^3$ is 2-connected, and in particular there is an $[a] \in H_2(X_d^2;\bZ)$ such that $\langle h, [a] \rangle=1$. Under the Thom isomorphism in the diagram above this corresponds to an $[\bar{a}] \in H_4(X_d^3, A_d^3 ; \bZ)$, and $\partial([\bar{a}])$ can be represented by a sphere in $A_d^3$ and hence a $[f] \in \pi_3(X_d^3)$ which is homologically trivial. By construction $\bar{a}$ is a 4-chain in $X_d^3$ which bounds $f$, and we can choose the disc filling $f$ to be in $\bA^4$. Thus the intersection of this cycle with $\bC\bP^2$ is the same as the intersection of $\bar{a} \cap \bC\bP^3 = a$ with $\bC\bP^1 \subset \bC\bP^3$, which is $\langle h, [a] \rangle=1$ by design.
\end{proof}

We now use the work of Pham \cite{Pham} describing $\pi_3(A_d^3) \cong H_3(A_d^3;\bZ)$, and its extension by Looijenga \cite{LooijengaFermat} describing $H_3(X_d^3;\bZ)$. These are described in terms of the action of the abelian group $\mu_d^4 = \langle t_1, t_2, t_3, t_4 \, | \, t_1^d, t_2^d, t_3^d, t_4^d \rangle$ of 4-tuples of $d$-th roots of unity, which acts on $\bC\bP^4$ by
$$t_1^{a_1} \cdots t_4^{a_4} \cdot [z_0 : z_1 : \cdots : z_4] = [z_0 : t_1^{a_1}z_1 : t_2^{a_2}z_2 : t_3^{a_3} z_3 : t_4^{a_4}z_4] $$
and hence acts on $X_d^3$, and by restriction acts on $\bA^4$ and hence $A_n^d$ in the evident way. Pham showed that there is a $\bZ[\mu_d^4]$-module isomorphism
$$H_3(A^3_d ; \bZ) \cong \bZ[\mu_d^4]/(\sum_{k=0}^{d-1} t_i^k, i = 1,2,3,4),$$
and Looijenga showed that the surjection $H_3(A_d^3;\bZ) \to H_3(X_d^3;\bZ)$ is the quotient by the sub-$\bZ[\mu_d^4]$-module $I$ generated by the element $\sum_{k=0}^{d-1} (t_1t_2t_3t_4)^k$.

With the discussion above we get a map of extensions of $\bZ[\mu_d^4]$-modules:
\begin{equation*}
\begin{tikzcd}[column sep=0.4cm]
0 \rar & I \rar \arrow[d, two heads] & \bZ[\mu_d^4]/(\sum_{k=0}^{d-1} t_i^k) \rar \arrow[d, two heads] & \bZ[\mu_d^4]/(\sum_{k=0}^{d-1} (t_1t_2t_3t_4)^k, \sum_{k=0}^{d-1} t_i^k) \rar \arrow[d, equals]& 0\\
0 \rar& \bZ/d\{\eta\} \rar & \pi_3(X_d^3) \rar & H_3(X_d ; \bZ) \rar & 0.
\end{tikzcd}
\end{equation*}
The induced map of long exact sequences on $H_*(\mu_d^4 ; -)$ contains the portion
\begin{equation*}
\begin{tikzcd}[column sep=0.5cm]
\cdots \rar & H_0(\mu_d^4 ; I) \rar \arrow[d, two heads] & \bZ/d \rar \arrow[d, two heads] & \bZ/d \rar \arrow[d, equals]& 0\\
\cdots\rar& \bZ/d\{\eta\} \rar & H_0(\mu_d^4 ; \pi_3(X_d^3)) \rar & H_0(\mu_d^4 ; H_3(X_d^3 ; \bZ)) \rar & 0.
\end{tikzcd}
\end{equation*}
The top right map is surjective and therefore bijective as the two groups have the same size: thus the top left map is zero: but as the vertical left map is onto it then follows that the bottom left map is zero: i.e.\, $\bZ/d\{\eta\} \to \pi_3(X_d^3) \to H_0(\mu_d^4 ; \pi_3(X_d^3))$ is zero.

The argument will be finished once we show that the automorphisms of $\pi_3(X_d^3)$ given by the $t_i \in \mu_d^4$ are realised by monodromy, as then the class $\eta$ vanishes in $H_0(\Mon_d ; \pi_3(X_d^3))$ too, as required. To see this, recall that $PGL_5(\bC)$ acts on the space $\mathcal{U}_d \subset \mathbb{P}H^0(\bC\bP^4 ; \mathcal{O}(d))$ of smooth hypersurfaces, and $\mu_d^4 \leq PGL_5(\bC)$ stabilises $X_d$: thus acting on $X_d$ gives a based map $PGL_5(\bC)/\mu_d^4 \to \mathcal{U}_d$ and hence a homomorphism $\pi_1(PGL_5(\bC)/\mu_d^4) \to \Mon_d$. The domain fits into an extension
$$1 \lra \pi_1(PGL_5(\bC)) = \bZ/5 \lra \pi_1(PGL_5(\bC)/\mu_d^4)  \lra \mu_d^4 \lra 1$$
and any lift of $t_i \in \mu_d^4$ to $\pi_1(PGL_5(\bC)/\mu_d^4)$ gives an element of the monodromy group which acts on $\pi_3(X_d)$ as $t_i$. (It might not have order $d$ in $\Mon_d$, but this is not important.)
\end{proof}

\section{Further upper bounds: cobordism considerations}\label{sec:FurtherUpperBounds}

The discussion so far gives the extension
\begin{equation*}
1 \lra \mathrm{K}_d \lra \Stab_{\MCG_d}(\ell^{\hyp}_{X_d}) \lra \mathrm{Aut}(\pi_3(X_d), \lambda, \mu) \lra 1
\end{equation*}
from \eqref{eq:StabExt}, and shows that the image of $\alpha: \Mon_d \to \Stab_{\MCG_d}(\ell^{\hyp}_{X_d}) $ surjects to $\mathrm{Aut}(\pi_3(X_d), \lambda, \mu)$ (Theorem \ref{thm:MonodromyOnPi3}) and contains the image of $\Phi : \Theta_7 \to \mathrm{K}_d$ (Theorem \ref{thm:Krylov}). To prove Theorem \ref{thm:Main} we must therefore account for the cokernel of $\Phi$. The goal of this section will be to make sense of, and prove, the statement that $\alpha: \Mon_d \to \Stab_{\MCG_d}(\ell^{\hyp}_{X_d})$ does not hit the $\mathrm{Coker}(\Phi)$ part of $\mathrm{K}_d$. We formulate this as follows.

\begin{theorem}\label{thm:6.1}
There is a factorisation
\begin{equation*}
\mathrm{K}_d \lra \Stab_{\MCG_d}(\ell^{\hyp}_{X_d}) \overset{\kappa}\lra \mathrm{Coker}(\Phi)
\end{equation*}
of the natural quotient map, such that $\kappa \circ \alpha$ is trivial.
\end{theorem}

This shows that $\alpha$ lands in $\mathrm{Ker}(\kappa: \Stab_{\MCG_d}(\ell^{\hyp}_{X_d}) \to \mathrm{Coker}(\Phi))$, which with our earlier discussion shows that this kernel is precisely $\mathrm{Im}(\alpha)$, giving an extension
\begin{equation*}
1 \lra \Phi(\Theta_7) \lra \mathrm{Im}(\alpha) \lra \mathrm{Aut}(\pi_3(X_d), \lambda, \mu) \lra 1,
\end{equation*}
which proves Theorem \ref{thm:Main}. 

The cokernel of $\Phi$ is described in Lemma \ref{lem:WhatIsKd} as
$$\mathrm{Coker}(\Phi) \cong \begin{cases}
\bZ/2 & d \equiv 0 \mod 4\\
0 & d \not\equiv 0 \mod 4
\end{cases} \oplus \begin{cases}
\bZ/3 & d \equiv 0 \mod 3\\
0 & d \not\equiv 0 \mod 3,
\end{cases}$$
so our strategy for proving Theorem \ref{thm:6.1} will be to analyse $H_1(\Stab_{\MCG_d}(\ell^{\hyp}_{X_d});\bZ)$, 2- and 3-locally. The method we will use is analogous to that of \cite{grwabelian}, though with significant additional difficulties.

\subsection{Calculating the abelianisation of $\Stab_{\MCG_d}(\ell^{\hyp}_{X_d})$}\label{sec:CalcAb}

Using the space $\Theta^{\hyp}(X_d)$ of $\theta^{\hyp}$-structures on $X_d$, we define
$$\MCG_d^{\hyp} := \pi_1(\Theta^{\hyp}(X_d) \hcoker \Diff(X_d), \ell^{\hyp}_{X_d}).$$
The long exact sequence on homotopy groups for this homotopy orbit space contains a portion
$$\pi_1(\Diff(X_d)) \lra \pi_1(\Theta^{\hyp}(X_d), \ell^{\hyp}_{X_d}) \lra \MCG_d^{\hyp} \lra \MCG_d \overset{\curvearrowright}\lra \theta^{\hyp}(X_d)$$
so gives a half-exact sequence
$$\pi_1(\Theta^{\hyp}(X_d), \ell^{\hyp}_{X_d}) \lra \MCG_d^{\hyp} \lra \Stab_{\MCG_d}(\ell^{\hyp}_{X_d}) \lra 1.$$

\subsubsection{Calculating the abelianisation of $\MCG_d^{\hyp}$}\label{sec:GRW}

The abelianisation of $\MCG_d^{\hyp}$ can be computed by cobordism-theoretic methods, using the results of \cite{grwcob, grwstab1, grwstab2} (proceeding similarly to \cite{grwabelian}). The result is given in terms of the corresponding unstable tangential structure defined by the homotopy cartesian square
\begin{equation}\label{eq:DefXid}
\begin{tikzcd}
B_6 \rar \dar{\theta^\hyp}& \bC\bP^\infty \dar{\bar{\theta}^{\hyp}}\\
B\OO(6) \rar& B\OO,
\end{tikzcd}
\end{equation}
specifically the Thom spectrum $MT\theta^\hyp(6) := Th(-(\theta^\hyp)^*\gamma_6)$ of the $(-6)$-dimensional virtual vector bundle $-(\theta^\hyp)^*\gamma_6$ over $B_6$.

\begin{proposition}
There is a map
\begin{equation*}
H_1(\MCG_d^{\hyp}; \bZ) \lra \pi_1^s(MT\theta^\hyp(6))
\end{equation*}
which is an isomorphism for $d \geq 3$.
\end{proposition}
\begin{proof}
The parameterised Pontrjagin--Thom construction gives a map 
\begin{equation}\label{eq:PTmap}
\Theta^{\hyp}(X_d) \hcoker \Diff(X_d) \lra \Omega^\infty MT\theta^\hyp(6).
\end{equation}
The lift $\ell^\hyp_{X_d} : X_d \to B_6$ of $\tau_{X_d} : X_d \to B\OO(6)$ is 3-connected, because $X_d \to \mathbb{CP}^4 \to \mathbb{CP}^\infty$ is by the Lefschetz hyperplane theorem, and $B_6 \to \mathbb{CP}^\infty$ is 6-connected by the homotopy cartesian square \eqref{eq:DefXid}. Thus \cite[Theorem 1.8]{grwstab2} applies to the path component of $\ell^\hyp_{X_d}$ in $\Theta^{\hyp}(X_d) \hcoker \Diff(X_d)$ and shows that the map \eqref{eq:PTmap} restricted to this path-component is an isomorphism on homology in degrees satisfying $2* \leq \bar{g}^\hyp(X_d, \ell^\hyp_{X_d})-3$. To obtain an isomorphism in degree 1, we must argue that $\bar{g}^\hyp(X_d, \ell^\hyp_{X_d}) \geq 5$.

The quantity $\bar{g}^\hyp(X_d, \ell^\hyp_{X_d})$ is described in \cite[Section 1.3]{grwstab2}, and is bounded below by ${g}^\hyp(X_d, \ell^\hyp_{X_d})$, the largest number of disjointly embedded copies of $W_{1,1} := (S^3 \times S^3) \setminus \mathrm{int}(D^6)$ into $X_d$ on which the restriction of $\ell^\hyp_{X_d}$ is ``admissible'' (\cite[Definition 1.3]{grwstab1}). We shall not need to go into the definition of admissible, because we can apply \cite[Remark 7.16]{grwstab1}: the intersection form $(H_3(X_d;\bZ), \lambda)$ is isomorphic to a sum of $g := \tfrac{d^4-5d^3+10d^2-10d + 4}{2}$ hyperbolic forms, and $H_3(B_6;\bZ)=0$, so by that remark ${g}^\hyp(X_d, \ell^\hyp_{X_d}) \geq g$. For $d \geq 3$ we have $g \geq 5$, and so indeed have $\bar{g}^\hyp(X_d, \ell^\hyp_{X_d}) \geq 5$.
\end{proof}

Proceeding in parallel to \cite[Section 5]{grwabelian}, the 6-connected map $B_6 \to \mathbb{CP}^\infty$ gives a 0-connected map
$$MT\theta^\hyp(6) \lra \Sigma^{-6} MT\bar{\theta}^\hyp$$
on Thomifying, where $MT\bar{\theta}^\hyp$ is the Thom spectrum of \emph{minus} the 0-dimensional virtual vector bundle classified by $\bar{\theta}^\hyp : \mathbb{CP}^\infty \to B\OO$. This produces a long exact sequence
$$\pi_8^s(MT\bar{\theta}^\hyp) \overset{\partial}\lra \bZ/4 \cong \pi_7^s(\SO/\SO(6)) \lra \pi_1^s(MT\theta^\hyp(6)) \lra \pi_7^s(MT\bar{\theta}^\hyp) \lra 0.$$
(The isomorphism and the 0 at the right-hand end follow from calculations of Paechter \cite{PaechterI} collected in \cite[Lemma 5.2]{grwabelian}: the 0 is because $\pi_7^s(MT\bar{\theta}^\hyp)$ is easily seen to be torsion and the next term is really $\pi_6^s(\SO/\SO(6)) \cong \bZ$.)

\subsubsection{Calculating $\pi_1(\Theta^{\hyp}(X_d), \ell^{\hyp}_{X_d})$}\label{sec:Reframing}

The group $\pi_1(\Theta^{\hyp}(X_d), \ell^{\hyp}_{X_d})$ can be approached by the same method as in Section \ref{sec:ThetaHyp}, which gives an exact sequence
$$\bZ = H^0(X_d;\bZ) \overset{(\theta^{\hyp})_*}\lra KO^{-2}(X_d) \lra \pi_1(\Theta^{\hyp}(X_d), \ell^{\hyp}_{X_d}) \lra H^1(X_d;\bZ) = 0.$$ 
Thus there is an exact sequence
$$KO^{-2}(X_d) \lra H_1(\MCG_d^{\hyp};\bZ) \lra H_1(\Stab_{\MCG_d}(\ell^{\hyp}_{X_d});\bZ) \lra 0.$$

Let us describe the composition
$$KO^{-2}(X_d) \lra H_1(\MCG_d^{\hyp};\bZ) \overset{\sim}\lra \pi_1^s(MT\theta^\hyp(6)).$$
In geometric terms it is given as follows. There is a map of vector bundles $\ell': T(S^1 \times X_d) \cong \underline{\bR} \oplus \pi_2^* TX_d \to \underline{\bR} \oplus (\theta^\hyp)^*\gamma_6$ induced by the $\theta^\hyp$-structure $\ell : TX_d \to (\theta^\hyp)^*\gamma_6$ on $X_d$, which via the Pontrjagin--Thom correspondence defines a class $[S^1 \times X_d , \ell'] \in \pi_1^s(MT\theta^\hyp(6))$. This is in fact $\eta \cdot [X_d,\ell]$. The map above is then given by
$$KO^{-2}(X_d) \overset{s \boxtimes -}\lra KO^{-1}(S^1 \times X_d) \lra \pi_1^s(MT\theta^\hyp(6))$$
where $s \in KO^1(S^1)$ is the suspension class, and the second map sends an element $[f : S^1 \times X_d \to \OO]$, with $f$ considered as a stable isomorphism of the vector bundle $T(S^1 \times X_d)$, to the class $[S^1 \times X_d, \ell' \circ f]$. In homotopy-theoretic terms it is given as follows. It sends the class $[g : S^1 \wedge (X_d)_+ \to \OO] \in KO^{-2}(X_d)$ to the class
\begin{equation*}
\begin{tikzcd}[column sep = 0.6cm]
S^1 \wedge S^0 \dar{S^1 \wedge [X_d]} & & &  MT\theta^\hyp(6) \\
S^1 \wedge Th(\nu_{X_d}) \arrow[r, "\substack{Thom \\ diag.}"] \arrow[rrr, bend right=20, "g_*"] & S^1 \wedge (X_d)_+ \wedge Th(\nu_{X_d}) \rar{g \wedge Id}\rar & \OO_+ \wedge Th(\nu_{X_d}) \rar{\text{act}} & Th(\nu_{X_d}). \uar{Th(\ell)}
\end{tikzcd}
\end{equation*}
Here $g_*$ is defined to be the lower composition, $[X_d] : S^0 \to Th(\nu_{X_d})$ is the Thom collapse map for $X_d$, and we have used the action of $\OO$ on all spectra, via $J : \OO \to GL_1(S^0)$.

\subsection{Cobordism calculation: 3-torsion}\label{sec:3tors}

We first treat the simpler case of the 3-torsion in $\Coker(\Phi)$, which occurs only when $d \equiv 0 \mod 3$. 

\begin{proposition}
Suppose $d \equiv 0 \mod 3$, $d \geq 3$. Then
\begin{enumerate}[(i)]
\item there is a surjection $p_3 : \pi_7^s(MT\bar{\theta}^\hyp)_{(3)} \to \bZ/3$, such that

\item the composition $KO^{-2}(X_d)_{(3)} \to \pi_7^s(MT\bar{\theta}^\hyp)_{(3)} \overset{p_3}\to \bZ/3$ is trivial.

\end{enumerate}
\end{proposition}

We calculate $\pi_7^s(MT\bar{\theta}^\hyp)_{(3)}$ using the Adams spectral sequence\footnote{The $E_2$-pages of the Adams spectral sequences in this and the following section were computed using the $\mathrm{Ext}$-calculator at \url{https://spectralsequences.github.io/sseq/}, and typeset using the {\tt spectralsequences} LaTeX package. Both are excellent.} at $p=3$. There is an isomorphism $H^*(MT\bar{\theta}^\hyp;\bF_3) = u \cdot H^*(\bC\bP^\infty;\bF_3)$, and the action of the Steenrod algebra is twisted. This is well-known at $p=2$, but less well known at odd primes so we briefly explain it. Write $\mathcal{P} = Id + \mathcal{P}^1 + \mathcal{P}^2 + \cdots$. If $u$ is the Thom class of a complex line bundle with Euler class $x$ then we have $\mathcal{P}(u) = u+u^p = u \cdot (1+x^{p-1})$ by the axioms of Steenrod operations and the definition of the Euler class. If $u$ is the Thom class of a sum of line bundles with Euler classes $x_1, x_2, \ldots, x_r$ then this gives
$$\mathcal{P}(u) = u \cdot \prod_{i=1}^r (1+x_i^{p-1}) = u \cdot \sum_{j=0}^r e_j(x_1^{p-1}, \ldots, x_r^{p-1})$$
where $e_j$ are the elementary symmetric polynomials. The expression $e_j(x_1^{p-1}, \ldots, x_r^{p-1})$ is again a symmetric polynomial, known as the $(p-1)$th Frobenius of $e_j$, so may be expressed in terms of the elementary symmetric polynomials $e_k(x_1, \ldots, x_r)$, i.e.\ the Chern classes of the original sum of line bundles. By the splitting principle this expression then holds for any complex vector bundle, and hence for any virtual bundle too.

In the case at hand we have $c(-\bar{\theta}^{\hyp}) = 1 + (d +1 ) x + d x^2 + x^3 + (1 + d) x^4 + O(x^5) \mod 3$, and  so
\begin{align*}
\mathcal{P}(u) &= u \cdot (1 + (c_1^2 - 2 c_2) + (c_2^2 - 2 c_3 c_1 + 2 c_4) + \cdots)\\
&= u \cdot (1 + (d^2+1)x^2 + d^2 x^4 + O(x^5)).
\end{align*}
In particular when $d \equiv 0 \mod 3$ we have $\mathcal{P}(u) = u \cdot (1 + x^2 + O(x^5))$. Then $H^*(MT\bar{\theta}^\hyp;\bF_3) = u \cdot H^*(\bC\bP^\infty;\bF_3)$ has generating Steenrod operations
$$\mathcal{P}^1(u) = u \cdot x^2, \quad \mathcal{P}^1(u \cdot x) = 2 u \cdot x^3$$
in degrees $\leq 9$. Using this we can calculate the $E^2$-page of the Adams spectral sequence in a range, which is shown in Figure \ref{fig:p3d0mod3}.

\begin{figure}[h]
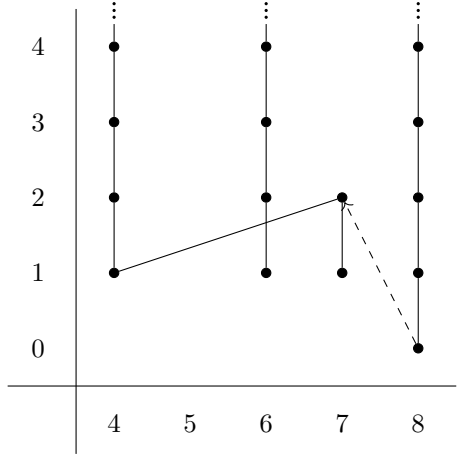

\begin{sseqdata}[ name = p3d0mod3, Adams grading, classes=fill, x range = {4}{8}, y range = {0}{4}]

\class(4,1)
\class(4,2)
\structline(4,1)(4,2)
\class(4,3)
\structline(4,2)(4,3)
\class(4,4)
\structline(4,3)(4,4)
\class(4,5)
\structline(4,4)(4,5)

\class(6,1)
\class(6,2)
\structline(6,1)(6,2)
\class(6,3)
\structline(6,2)(6,3)
\class(6,4)
\structline(6,3)(6,4)
\class(6,5)
\structline(6,4)(6,5)

\class(7,1)
\class(7,2)
\structline(7,1)(7,2)
\structline(4,1)(7,2)

\class(8,0)
\class(8,1)
\structline(8,0)(8,1)
\class(8,2)
\structline(8,1)(8,2)
\class(8,3)
\structline(8,2)(8,3)
\class(8,4)
\structline(8,3)(8,4)
\class(8,5)
\structline(8,4)(8,5)

\d[dashed]2(8,0)(7,2)
\end{sseqdata}
\printpage[ name = p3d0mod3, page = 2 ]
\caption{$E_2$-page of the Adams spectral sequence at $p=3$ for $MT\bar{\theta}^\hyp$, with $d \equiv 0 \mod 3$.}
\label{fig:p3d0mod3}
\end{figure}

We see that $\pi_7^s(MT\bar{\theta}^\hyp)_{(3)}$ is either $\bZ/9$ or $\bZ/3$, depending on the differential coming out of the 8-column, but no smaller. In either case, taking the quotient by those classes of $\bF_3$-Adams filtration $\geq 2$ gives a surjection
$$p_3 : \pi_7^s(MT\bar{\theta}^\hyp)_{(3)} \lra \bZ/3.$$

We now wish to show that the composition $KO^{-2}(X_d)_{(3)} \to \pi_7^s(MT\bar{\theta}^\hyp)_{(3)} \overset{p_3}\to \bZ/3$ is trivial. To do so, we make use of the following types of classes in $KO^{-2}(X_d)$.

\begin{lemma}\label{lem:ClassTypes}
The abelian group $KO^{-2}(X_d)$ is generated by the following classes:
\begin{enumerate}[(i)]
\item The pullback along the map $c : X_d \to S^6$ that collapses the complement $X'_d \subset X_d$ of a ball of the Bott class $\Sigma^{-2} \beta \in KO^{-2}(S^6)$,

\item classes pulled back along $\ell^\hyp_{X_d} : X_d \to \mathbb{CP}^\infty$.
\end{enumerate}
\end{lemma}
\begin{proof}
Clearly the image of $KO^{-2}(*) \to KO^{-2}(X_d)$ is of type (ii). The Atiyah--Hirzebruch spectral sequence gives an extension
$$0\lra H^6(X_d ; \bZ) \lra \widetilde{KO}^{-2}(X_d)\lra H^2(X_d ; \bZ) \lra 0$$
and the left-hand term corresponds to the classes of type (i). The map of Atiyah--Hirzebruch spectral sequences for $\ell^\hyp_{X_d} : X_d \to \mathbb{CP}^\infty$ shows that the right-hand term can be saturated by classes of type (ii).
\end{proof}

For classes of type (i), we consider the commutative diagram
\begin{equation*}
\begin{tikzcd}[column sep=0.3cm]
\widetilde{KO}^{-2}(S^6) \dar{\cong}\\
{KO}^{-2}(X_d, X'_d) \rar \dar& \pi_1^s(Th(\nu_{X_d}), Th(\nu_{X_d}\vert_{X'_d})) \dar \arrow[r, equals]& \pi_7^s(S^0) \dar \arrow[rd, "\iota"]\\
{KO}^{-2}(X_d) \rar{}& \pi_1^s(Th(\nu_{X_d})) \rar{Th(\ell)_*}& \pi_1^s(MT{\theta}^\hyp(6)) \rar & \pi_7^s(MT\bar{\theta}^\hyp).
\end{tikzcd}
\end{equation*}
Recall\footnote{This follows because Adams' $e$-invariant map $e'_\mathbb{R} : \pi_7^s(S^0) \to \bQ/\bZ$ is an isomorphism onto $\bZ[\tfrac{1}{240}]/\bZ$ \cite[p.\ 46]{AdamsJX4}, and the octonionic Hopf fibration is the sphere bundle of the tautological octonionic line bundle over $S^8 = \mathbb{OP}^1$, whose reduced class generates $\widetilde{KO}^0(S^8)$, so $\sigma = \pm j_7$ in the notation of Adams' paper and so $e'_\bR(\sigma) = \pm \tfrac{1}{240}$ by \cite[Example 7.17]{AdamsJX4}.} that the octonionic Hopf fibration $S^{15} \to S^8$ represents a class $\sigma \in \pi_7^s(S^0)$,
and that $\pi_7^s(S^0) = \bZ/240\{\sigma\}$. Hence the diagram shows that $c^*(\Sigma^{-2}\beta)$ maps to a multiple of $\iota \circ \sigma$ in $\pi_7^s(MT\bar{\theta}^\hyp)$, but $\sigma \in \pi_7^s(S^0)$ has $\bF_3$-Adams filtration 2 so $\iota \circ \sigma$ has $\bF_3$-Adams filtration $\geq 2$, but this is precisely what we divided out to form the quotient $p_3 : \pi_7^s(MT\bar{\theta}^\hyp)_{(3)} \to \bZ/3$.

A class $[g] \in {KO}^{-2}(X_d)$ of type (ii) is by definition pulled back from some $[G] \in {KO}^{-2}(\bC\bP^\infty)$. By the construction at the end of Section \ref{sec:Reframing} this induces a map
\begin{equation*}
G_* : S^1 \wedge MT\bar{\theta}^\hyp \lra MT\bar{\theta}^\hyp
\end{equation*}
and by naturality the image of $[g] \in {KO}^{-2}(X_d)$ in $\pi_7^s(MT\bar{\theta}^\hyp)$ is given by
$$S^1 \wedge S^6 \overset{S^1 \wedge [X_d, \ell^\hyp_{X_d}]}\lra S^1 \wedge MT\bar{\theta}^\hyp \overset{G_*} \lra MT\bar{\theta}^\hyp,$$
for $[X_d, \ell^\hyp_{X_d}] \in \pi_6^s(MT\bar{\theta}^\hyp)$.

In the Adams chart for $MT\bar{\theta}^\hyp$ in Figure \ref{fig:p3d0mod3} we see that $[X_d, \ell^\hyp_{X_d}] \in \pi_6^s(MT\bar{\theta}^\hyp)$ has $\bF_3$-Adams filtration $\geq 1$. Furthermore, the map $G_*$ is trivial on $\bF_3$-cohomology, so has $\bF_3$-Adams filtration $\geq 1$. It follows that the image of $[g]$ in $\pi_7^s(MT\bar{\theta}^\hyp)$ has $\bF_3$-Adams filtration $\geq 2$, but again this is precisely what we divided out to form the quotient $p_3 : \pi_7^s(MT\bar{\theta}^\hyp)_{(3)} \to \bZ/3$.

\subsection{Cobordism calculation: 2-torsion}\label{sec:2tors}

We now treat the 2-torsion in $\Coker(\Phi)$, which occurs only when $d \equiv 0 \mod 4$. It is parallel to the 3-torsion case just described, but somewhat more complicated. 

\begin{proposition}
Suppose $d \equiv 0 \mod 4$, $d \geq 4$. Then
\begin{enumerate}[(i)]
\item there is a surjection $p_2 : \pi_7^s(MT\bar{\theta}^\hyp)_{(2)} \to \bZ/2$, such that

\item the composition $KO^{-2}(X_d)_{(2)} \to \pi_7^s(MT\bar{\theta}^\hyp)_{(2)} \overset{p_2}\to \bZ/2$ is trivial.

\end{enumerate}
\end{proposition}

To understand the 2-local homotopy groups of $MT\bar{\theta}^\hyp$, we can identify the homotopy type of its 9-skeleton with a certain stunted complex projective space.

\begin{lemma}\label{lem:JamesPeriodicity}
If $d \equiv 0 \mod 8$ then there is a 2-local equivalence
$$Th(\mathcal{O}(d) - 5 \mathcal{O}(1)+4 \to \bC\bP^4) \simeq \Sigma^{-2(2^6-5)}\bC\bP_{2^6-5}^{2^6-1}.$$
If $d \equiv 4 \mod 8$ then there is a 2-local equivalence
$$Th(\mathcal{O}(d) - 5 \mathcal{O}(1)+4 \to \bC\bP^4) \simeq \Sigma^{-2(2^5-5))}\bC\bP_{2^5-5}^{2^5-1}.$$
\end{lemma}
\begin{proof}
We claim that the 2-local spherical fibrations for the complex bundles $\mathcal{O}(d) - 5 \mathcal{O}(1)+4$ and $(2^6-5)(\mathcal{O}(1)-1)$ or $(2^5-5)(\mathcal{O}(1)-1)$ are equivalent: then the 2-local Thom spaces are equivalent, and it is standard that $Th(k\mathcal{O}(1) \to \bC\bP^n) \simeq \bC\bP_k^{n+k}$. This is a kind of James periodicity, and comes down to showing these bundles agree under the 2-local $J$-homomorphism
$$J_{(2)} : KO^0(\bC\bP^4) \lra J(\bC\bP^4)_{(2)}.$$
The latter groups have been calculated by Adams and Walker \cite{AdamsWalker}, and we explain how to extract the specific information that we want.

Writing $x := [\mathcal{O}(1)]-1 \in \widetilde{K}^0(\bC\bP^4)$ and $y := r(x) \in \widetilde{KO}^0(\bC\bP^4)$, by \cite[Section 2]{AdamsWalker} we have $K^0(\bC\bP^4) = \bZ[x]/(x^5)$ and ${KO}^0(\bC\bP^4) = \bZ[y]/(y^3)$. The complexification map $c : KO^0(\bC\bP^4) \to K^0(\bC\bP^4)$ is a ring homomorphism and commutes with Adams operations. It satisfies 
\begin{align*}
c(y) &= cr(x) = x + \psi^{-1}_\bC(x)\\
&= x + (1+x)^{-1} - 1 = \frac{x^2}{1+x}\\
&= x^2 - x^3 + x^4
\end{align*}
and so $c(y^2) = x^4$, and hence $c$ is injective. From this one can check that $r(x^2) = 2y + y^2$, so $r : K^0(\bC\bP^4) \to KO^0(\bC\bP^4)$ is surjective. As $K^0(\bC\bP^4)$ is spanned by sums of complex line bundles, it follows that ${KO}^0(\bC\bP^4)$ is spanned by sums of $\OO(2)$-bundles, and so an easy version of the Adams conjecture \cite[Theorem 1.3]{AdamsJX1} applies to show that for any odd $k$ and any $z \in KO^0(\bC\bP^4)$ the class $(\psi^k_\bR - 1)z$ is in the kernel of $J_{(2)}$. Let us write $a \sim b$ to mean that $a-b \in KO^0(\mathbb{CP}^4)$ lies in the kernel of $J_{(2)}$.

As 
$$\psi_\bC^k(x) = (1+x)^k - 1 = kx + \binom{k}{2} x^2 + \binom{k}{3} x^3 + \binom{k}{4} x^4$$
and $c$ is injective and commutes with Adams operations, we have
\begin{align*}
c(\psi_\bR^k(y)) &= \psi^k_\bC(x^2 -x^3+x^4)\\
&= k^2 x^2 - k^2 x^3 + \tfrac{1}{12} k^2 (k^2 + 11) x^4\\
&= c(k^2y + \tfrac{1}{12} k^2 (k^2 -1)y^2)
\end{align*}
giving $\psi_\bR^k(y) = k^2y + \tfrac{1}{12} k^2 (k^2 -1)y^2$, and hence $\psi_\bR^k(y^2) = k^4 y^2$. In particular we have
\begin{equation*}
(\psi_\bR^3-1)(y) = 2^3 y + 2 \cdot 3y^2 \quad\quad (\psi_\bR^3-1)(y^2) = 2^4 \cdot 5 y^2.
\end{equation*}
As the target of $J_{(2)}$ is 2-local, it follows that $2^4 y^2 \sim 0$, and hence that $2^6 y \sim 0$.

By \cite[Lemma A.2]{AdamsWalker} realification commutes with Adams operations, so writing $d=2^s \cdot t$ with $t$ odd and $s \geq 2$, we have
\begin{align*}
r(\mathcal{O}(d)-1) &= r(\psi^d_\bC(\mathcal{O}(1)-1)) = \psi_\bR^d(y)\\
&= \psi_\bR^{t} (\psi_\bR^{2^s}(y)) \\
&\sim \psi_\bR^{2^s}(y) =  2^{2s} y +  2^{2s-2} \tfrac{(2^{2s} -1)}{3}y^2
\end{align*}
Suppose first that $s \geq 3$. Then $2^{2s-2}y^2 \sim 0$, and $2^{2s} y \sim 0$, so $r(\mathcal{O}(d)-1) \sim 0 \sim 2^6 y$, and so $r(\mathcal{O}(d) - 5 \mathcal{O}(1)+4) \sim r((2^6-5)(\mathcal{O}(1)-1))$ as required.

If $s=2$ then $\tfrac{2^{2s}-1}{3} = 5$ so the above is
\begin{align*}
2^{4} y +  2^{2} \cdot 5 y^2 &= 2^5 y - 2^4 y +  2^{2} \cdot 5 y^2\\
 &\sim 2^5 y - 2(-2 \cdot 3 y^2)  +  2^{2} \cdot 5 y^2\\
 &= 2^5 y + 2^5 y^2 \sim 2^5 y
\end{align*}
and so $r(\mathcal{O}(d) - 5 \mathcal{O}(1)+4) \sim r((2^5-5)(\mathcal{O}(1)-1))$ as required.
\end{proof}

\begin{corollary}\label{cor:2LocHtyGps}
If $d \equiv 0 \mod 4$ then
\begin{align*}
\pi_5(MT\bar{\theta}^\hyp)_{(2)} &\cong 
\bZ/4\\
\pi_7(MT\bar{\theta}^\hyp)_{(2)} &\cong \begin{cases}
\bZ/2 \oplus \bZ/16 & d \equiv 0 \mod 8\\
\bZ/2 \oplus \bZ/8 & d \equiv 4 \mod 8.
\end{cases}\\
\pi_8(MT\bar{\theta}^\hyp)_{(2)} &\cong \bZ_{(2)} \oplus \bZ/4.
\end{align*}
\end{corollary}
\begin{proof}
We have $MT\bar{\theta}^\hyp = Th(\mathcal{O}(d) - 5 \mathcal{O}(1) + 4 \to \bC\bP^\infty)$, and by Lemma \ref{lem:JamesPeriodicity} its 9-skeleton is 2-locally equivalent to either $\Sigma^{-2(2^6-5)}\bC\bP_{2^6-5}^{2^6-1}$ or $\Sigma^{-2(2^6-5)}\bC\bP_{2^6-5}^{2^6-1}$, which in turn are the 9-skeletons of $\Sigma^{-2(2^6-5)}\bC\bP_{2^6-5}^{\infty}$ or $\Sigma^{-2(2^6-5)}\bC\bP_{2^6-5}^{\infty}$. 

The first claim follows from \cite[Theorem 1 c)]{Matsunaga1}, using the discussion in \S 9 of that paper and Toda's identification $\pi_{2n+i}(U(n)) \cong \pi_i^s(\bC\bP_{n}^\infty)$ for odd $i$ in the metastable range. The second claim follows from \cite[Theorem 2]{Matsunaga2}. The third claim follows from \cite[Table 2.2]{MosherStunted}.
\end{proof}

On the other hand, we could also approach the homotopy groups of $MT\bar{\theta}^\hyp$ via the Adams spectral sequence. As modules over the Steenrod algebra we have $H^*(MT\bar{\theta}^\hyp;\bF_2) \cong u \cdot H^*(\bC\bP^\infty;\bF_2)$ where $H^*(\bC\bP^\infty;\bF_2)$ carries its usual Steenrod-module structure, and $\Sq(u) = u \cdot w(-\theta^{\hyp})$, which as $d$ is even is 
$$\Sq(u) = u \cdot (1 + x + x^2 + x^3 + O(x^5)).$$
In degrees $\leq 9$ we then have that $H^*(MT\bar{\theta}^\hyp;\bF_2)$ is a sum of modules over the Steenrod algebra $\bF_2\{u, u \cdot x, u \cdot x^2, u \cdot x^3\} \oplus \bF_2\{u \cdot x^4\}$ with generating Steenrod operations
$$\Sq^2(u) = u \cdot x, \quad \Sq^4(u) = u \cdot x^2, \quad \Sq^2(u \cdot x^2) = u \cdot x^3.$$
The $E_2$-page of the Adams spectral sequence near the degrees in which we are interested in then as in Figure \ref{fig:p2deven}.

\begin{figure}[h]
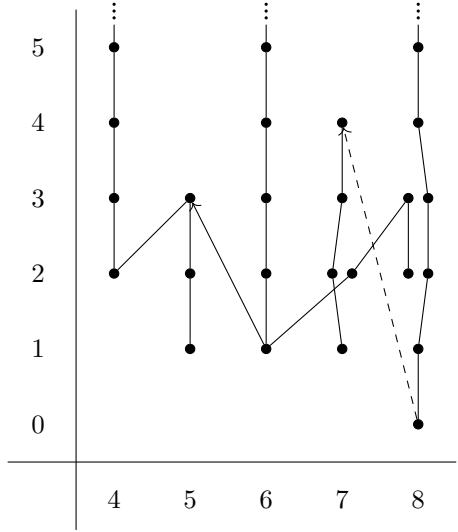


\begin{sseqdata}[ name = p2deven, Adams grading, classes=fill, x range = {4}{8}, y range = {0}{5}]

\class(4,2)
\class(4,3)
\structline(4,2)(4,3)
\class(4,4)
\structline(4,3)(4,4)
\class(4,5)
\structline(4,4)(4,5)
\class(4,6)
\structline(4,5)(4,6)

\class(5,1)
\class(5,2)
\structline(5,1)(5,2)
\class(5,3)
\structline(5,2)(5,3)
\structline(4,2)(5,3)

\class(6,1)
\class(6,2)
\structline(6,1)(6,2)
\class(6,3)
\structline(6,2)(6,3)
\class(6,4)
\structline(6,3)(6,4)
\class(6,5)
\structline(6,4)(6,5)
\class(6,6)
\structline(6,5)(6,6)

\class(7,1)
\class(7,2)
\class(7,2)
\structline(7,1)(7,2,1)
\structline(6,1)(7,2,2)
\class(7,3)
\structline(7,2,1)(7,3)
\class(7,4)
\structline(7,3)(7,4)

\class(8,0)
\class(8,1)
\structline(8,0)(8,1)
\class(8,2)
\class(8,2)
\structline(8,1)(8,2,2)
\class(8,3)
\class(8,3)
\structline(8,2,1)(8,3,1)
\structline(8,2,2)(8,3,2)
\structline(7,2,2)(8,3,1)
\class(8,4)
\structline(8,3,2)(8,4)
\class(8,5)
\structline(8,4)(8,5)
\class(8,6)
\structline(8,5)(8,6)

\d2(6,1)(5,3)
\d[dashed]4(8,0)(7,4)
\end{sseqdata}
\printpage[ name = p2deven, page = 2--4 ]

\caption{$E_2$-page of the Adams spectral sequence at $p=2$ for $MT\bar{\theta}^\hyp$, with $d$ even. The class of filtration 1 in degree 7 detects the image of $\sigma \in \pi_7^s(S^0)$ on the bottom cell $\iota : S^0 \to MT\bar{\theta}^\hyp$.}\label{fig:p2deven}
\end{figure}

From the homotopy groups shown in Corollary \ref{cor:2LocHtyGps}, we see that there is a unique pattern of differentials in this range:  $d_2 : E_2^{1,7} \to E_2^{3,8}$ is an isomorphism, and $d_4 : E_4^{0,8} \to E_4^{4,11}$ is an isomorphism if and only if $d \equiv 4 \mod 8$, with all other differentials being zero. In particular, taking the quotient by the subgroup generated by $\iota \circ \sigma \in \pi_7^s(MT\bar{\theta}^\hyp)_{(2)}$ gives a surjection
$$p_2 : \pi_7^s(MT\bar{\theta}^\hyp)_{(2)} \lra \bZ/2.$$

We wish to show that the composition $KO^{-2}(X_d)_{(2)} \to \pi_1^s(MT\bar{\theta}^\hyp)_{(2)} \overset{p_2}\to \bZ/2$ is trivial. We do this just as in the 3-torsion case, using Lemma \ref{lem:ClassTypes}. Namely, classes of type (i) coming from $c : X_d \to S^6$ map to a multiple of $\iota \circ \sigma \in \pi_7^s(MT\bar{\theta}^\hyp)$ so vanishes under $p_2$. As in the 3-torsion case, for classes of type (ii) their images may be represented as
$$S^1 \wedge S^6 \overset{S^1 \wedge [X_d,\ell^\hyp_{X_d}]}\lra S^1 \wedge MT\bar{\theta}^\hyp \overset{G_*}\lra MT\bar{\theta}^\hyp$$
for certain maps $G_*$. Figure \ref{fig:p2deven} shows the Adams chart for $MT\bar{\theta}^\hyp$, and because of the differential out of the 6-column, the class $[X_d,\ell^\hyp_{X_d}] \in \pi_6^s(MT\bar{\theta}^\hyp) $ has $\bF_2$-Adams filtration $\geq 2$. The map $G_*$ has $\bF_2$-Adams filtration $\geq 1$, so the classes obtained in this way all have $\bF_2$-Adams filtration $\geq 3$. Thus they lie in the subgroup of $\pi_7^s(MT\bar{\theta}^\hyp)_{(2)}$ generated by $4 \iota \circ \sigma$, and so vanish under $p_2$. 

\subsection{Proof of Theorem \ref{thm:6.1}}

Firstly, abbreviate $\mathrm{Aut} := \mathrm{Aut}(\pi_3(X_d), \lambda, \mu)$ and consider the Serre spectral sequence for \eqref{eq:StabExt}, which gives
\begin{equation}\label{eq:SS}
H_2(\mathrm{Aut};\bZ) \overset{d_2}\lra \mathrm{K}_d \lra H_1(\Stab_{\MCG_d}(\ell^{\hyp}_{X_d});\bZ) \lra H_1(\mathrm{Aut};\bZ) \lra 0.
\end{equation}

The group $\mathrm{Aut}$ participates in an extension
$$1 \lra H^3(X_d ; 2 \cdot \bZ/d) \lra \mathrm{Aut} \overset{\rho} \lra \begin{cases}
\mathrm{Aut}(H_3(X_d;\bZ), \lambda, \mu) & d \text{ odd}\\
\mathrm{Aut}(H_3(X_d;\bZ), \lambda) & d \text{ even}
\end{cases} \lra 1,$$
where the (outer) action of the quotient on the kernel is the usual one. Writing $\mathrm{Aut}'$ for the right-hand term, the Serre spectral sequence for this extension gives
$$H_0(\mathrm{Aut}' ; H^3(X_d ; 2 \cdot \bZ/d)) \lra H_1(\mathrm{Aut};\bZ) \lra H_1(\mathrm{Aut}' ; \bZ) \lra 1.$$
Using that $d \geq 3$ and so $g \geq 5$, the right-hand term is $\bZ/4$ if $d$ is odd and $0$ if $d$ is even by e.g.\ \cite[Section 4.1.1]{KR-WFram}. Similarly, if $d$ is even, or $d$ is odd and $\mu$ has Arf invariant 0, then the left-hand term vanishes by \cite[Lemma A.2]{KrannichMCG}. If $d$ is odd and $\mu$ has Arf invariant 1 then it still vanishes, though we do not know a specific reference: a similar argument to \cite[Lemma A.2]{KrannichMCG} works. The overall conclusion is that $H_1(\mathrm{Aut};\bZ)$ is $\bZ/4$ if $d$ is odd and $0$ if $d$ is even.

By the discussion in Sections \ref{sec:CalcAb}, \ref{sec:3tors}, and \ref{sec:2tors}, for $d \geq 3$ we have the diagram
\begin{equation*}
\begin{tikzcd}
KO^{-2}(X_d) \rar \arrow[rdd] & H_1(\MCG_d^\hyp;\bZ) \arrow[d, "\cong"] \rar & H_1(\Stab_{\MCG_d}(\ell^{\hyp}_{X_d});\bZ) \rar \arrow[dd, dashed, two heads, "\kappa"] &0\\
 & \pi_1^s(MT{\theta}^\hyp(6)) \arrow[d, two heads]\\
 & \pi_7^s(MT\bar{\theta}^\hyp) \arrow[r, two heads, "p_2 \oplus p_3"] & \mathrm{Coker}(\Phi)'
\end{tikzcd}
\end{equation*}
where the top row is exact and the lower composition is trivial, so the dashed surjection $\kappa$ exists. Here we have written $\mathrm{Coker}(\Phi)'$ for the abstract group
$$\begin{cases}
\bZ/2 & d \equiv 0 \mod 4\\
0 & d \not\equiv 0 \mod 4
\end{cases} \oplus \begin{cases}
\bZ/3 & d \equiv 0 \mod 3\\
0 & d \not\equiv 0 \mod 3.
\end{cases}$$

\begin{lemma}\label{lem:kappaZeroOnAlpha}
The composition 
$$\Mon_d \overset{\alpha}\lra \Stab_{\MCG_d}(\ell^{\hyp}_{X_d}) \lra H_1(\Stab_{\MCG_d}(\ell^{\hyp}_{X_d});\bZ) \overset{\kappa}\lra \mathrm{Coker}(\Phi)'$$
is trivial.
\end{lemma}
\begin{proof}
$\mathcal{U}_d \subset \mathbb{P}H^0(\mathbb{CP}^4 ; \mathcal{O}(d))$ is the complement of the discriminant locus. The discriminant locus is irreducible, and a generic point of it is a hypersurfaces which is a smooth except for a single ordinary double point (see proof of \cite[Chapter 1 Theorem 2.2]{Huybrechts}). It follows that its fundamental group $\Mon_d$ is normally generated by a single element, which can be taken to be a symplectic Dehn twist $\tau$ \cite[Chapter 2, \S 1.3]{ArnoldBook}. As $\mathrm{Coker}(\Phi)$ is an abelian group, it is enough to show that this Dehn twist maps to zero. The map $\alpha$ naturally lifts to $\MCG_d^\hyp$, as the universal family of smooth hypersurfaces admits a fibrewise $\theta^\hyp$-structure as discussed in Section \ref{sec:ThetaHyp}.

By choosing a degeneration of $X_d$ to an $A_2$-singularity, we may find an embedding $T^* S^3 \natural T^* S^3 \subset X_d$ of the plumbing such that $\tau$ is the Dehn twist around one of the spheres. There is a diffeomorphism $T^* S^3 \natural T^* S^3 \cong (S^3 \times S^3) \setminus D^6 = \mathrm{int}(W_{1,1})$ and so the mapping torus $T_\tau$ can be expressed as
$$T_\tau = T_{\tau\vert_{W_{1,1}}} \bigcup_{S^1 \times S^5} (S^1 \times (X_d \setminus W_{1,1})),$$
equipped with a $\theta^\hyp$-structure $\ell^\hyp_{T_\tau}$ that agrees with the pullback of $\ell^\hyp_{X_d}$ on the second term. Via $D^2 \times (X_d \setminus W_{1,1})$, the second term is $\theta^\hyp$-cobordant rel boundary to $D^2 \times S^5$, so $T_\tau$ is $\theta^\hyp$-cobordant to the manifold
$$M := T_{\tau\vert_{W_{1,1}}} \bigcup_{S^1 \times S^5} D^2 \times S^5$$
with some $\theta^\hyp$-structure $\ell_M^\hyp : M \to \mathbb{CP}^\infty$. But $M$ is easily checked to be 2-connected, so $\ell^\hyp_M$ may be lifted along $E\OO \to \mathbb{CP}^\infty$ to a stable framing $\ell_M^\sfr$. It follows that $[T_\tau,\ell^\hyp_{T_\tau}] \in \pi_7^s(MT\bar{\theta}^\hyp)$ lies in the image of $\iota_* : \pi_7^s(S^0)  \to \pi_7^s(MT\bar{\theta}^\hyp)$. Both 3-locally and 2-locally the image of this map lies in the subgroup that we divided out to form the map $p_2 \oplus p_3 : \pi_7^s(MT\bar{\theta}^\hyp)_{(2)} \to \mathrm{Coker}(\Phi)'$.
\end{proof}

The argument in the proof of the preceding lemma can be understood in a more general context by the discussion in Section \ref{sec:Wg1}, see Remark \ref{rem:WgAbApplication}.

\begin{lemma}
The map $q: \mathrm{K}_d \to H_1(\Stab_{\MCG_d}(\ell^{\hyp}_{X_d});\bZ) \overset{\kappa}\to \mathrm{Coker}(\Phi)'$ is the quotient by $\mathrm{Im}(\Phi)$.
\end{lemma}
\begin{proof}
For this it suffices to show that it is 2- and 3-locally surjective, and that it vanishes on precomposing with $\Phi : \Theta_7 \to \mathrm{K}_d$. Working 2-locally there is only anything to show when $d \equiv 0 \mod 4$, in which case $H_1(\mathrm{Aut};\bZ)=0$ and so $\mathrm{K}_d \to H_1(\Stab_{\MCG_d}(\ell^{\hyp}_{X_d});\bZ)$ is surjective, so $q$ is surjective too. Working 3-locally there is only anything to show when $d \equiv 0 \mod 3$. Then $H_1(\mathrm{Aut};\bZ)$ is 0 or $\bZ/4$ so is 3-locally trivial and so $\mathrm{K}_d \to H_1(\Stab_{\MCG_d}(\ell^{\hyp}_{X_d});\bZ)$ is 3-locally surjective, and so $q$ is 3-locally surjective too.

To see that $\Theta_7 \overset{\Phi}\to \mathrm{K}_d \overset{q}\to \mathrm{Coker}(\Phi)'$ is trivial, we use that $\Phi(\Theta_7)$ lies in the image of $\alpha$ by Theorem \ref{thm:Krylov}, then apply Lemma \ref{lem:kappaZeroOnAlpha}.
\end{proof}

\begin{remark}[Complete intersections]\label{rem:CI}
Many of the steps in our argument have more or less obvious analogues for a smooth 3-dimensional complete intersection $X_{d_1, \ldots, d_r} \subset \mathbb{CP}^{3+r}$ residing in the moduli space $\mathcal{U}_{d_1, \ldots, d_r}$ of such. For example, the results of Beauville and of Kreck--Su are still available. But as many of the arguments require careful calculations, it is hard to know whether to expect an analogous answer. 

For the reader interested in pursuing this, we briefly comment on what we perceive to be the difficulties. One should be able to construct a quadratic refinement on $\pi_3(X_{d_1, \ldots, d_r})$ as in Section \ref{sec:SurgeryKernel}, but a new argument will be needed for the analogue of Lemma \ref{lem:QuadDescends} because of the more complicated form of \cite[Théorème 6]{Beauville}. The analogue of Theorem \ref{thm:Krylov} will still hold as $X_{d_1, \ldots, d_r}$ still admits a deformation to an $E_6$-singularity. We do not know what should replace the results of Pham and of Looijenga in Section \ref{sec:AutMonodromy}: this seems like a serious problem. The analogue of Section \ref{sec:FurtherUpperBounds} is unpredictable: if $\mathrm{Coker}(\Phi : \Theta_7 \to \mathrm{K}_{d_1, \ldots, d_r})$ cannot be detected in a similar way, which depends on the results of many delicate calculations, then we expect it will be very hard to decide whether this cokernel can be realised by monodromy.
\end{remark}

\section{Proof of Theorem \ref{thm:B}: Comparison to $W_{g,1}$}\label{sec:Wg1}

In order to prove Theorem \ref{thm:B} we will compare some of our calculations with analogous calculations for mapping class groups of connect-sums of $S^3 \times S^3$'s, which have been studied in some detail, especially recently. Recall from Section \ref{sec:TopOfHyp} that by a theorem of Wall we may find an embedding
$$e : W_{g,1} := (\#^g S^3 \times S^3) \setminus \mathrm{int}(D^6) \lra X_d$$
which induces an isomorphism on $H_3(-;\bZ)$. By our assumption $d \geq 3$ we have $g \geq 5$. Let $\ell^{\hyp}_{W_{g,1}} = e^* \ell^{\hyp}_{X_d}$ be the induced $\theta^{\hyp}$-structure, which in particular induces a $\theta^{\hyp}$-structure $\ell^{\hyp}_{\partial W_{g,1}}$ on the boundary $\partial W_{g,1}$. Recall that the universal principal $\OO$-bundle $\bar{\theta}^{\sfr} : E\OO \to B\OO$ defines a stable tangential structure that we call a \emph{stable framing}.

\begin{lemma}\label{lem:SFReqHYP}
There is a boundary condition on stable framings $\ell^{\sfr}_{\partial W_{g,1}}$ and a $\Diff_\partial(W_{g,1})$-equivariant homotopy equivalence
$$u_* : \Theta^{\sfr}(W_{g,1} ; \ell^{\sfr}_{\partial W_{g,1}}) \overset{\sim}\lra \Theta^{\hyp}(W_{g,1} ; \ell^{\hyp}_{\partial W_{g,1}}).$$
\end{lemma}
\begin{proof}
The map $\ell^{\hyp}_{\partial W_{g,1}} : \partial W_{g,1} = S^5 \to \bC\bP^\infty$ is nullhomotopic, giving a lift along $u : E\OO \to \bC\bP^\infty$ and hence a boundary condition $\ell^{\sfr}_{\partial W_{g,1}}$ for stable framings (and a canonical isomorphism from $u_* \ell^{\sfr}_{\partial W_{g,1}}$ to $\ell^{\hyp}_{\partial W_{g,1}}$). The map $u$ then induces a map $u_*$ between the two spaces of structures. Seeing that it is surjective in $\pi_0$ means showing that the relative lifting problem
\begin{equation*}
\begin{tikzcd}
\partial W_{g,1} \dar \rar{\ell^{\sfr}_{\partial W_{g,1}}} & E\OO \dar{u}\\
W_{g,1} \rar{\ell^{\hyp}} \arrow[ru, dashed] & \bC\bP^\infty
\end{tikzcd}
\end{equation*}
can always be solved: it can because $H^2(W_{g,1}, \partial W_{g,1};\bZ)=0$. Moreover it can be solved uniquely up to homotopy as all the lower relative cohomology groups also vanish, which shows that the fibres of $u_*$ are contractible too.
\end{proof}

This lemma means that $\ell^{\hyp}_{W_{g,1}}$ corresponds to a canonical stable framing $\ell^{\sfr}_{W_{g,1}}$. In Section \ref{sec:SurgeryKernel} we described a quadratic refinement $\mu = \mu_{\ell^{\hyp}_{X_d}}$ on $\pi_3(X_d)$, which restricts along $e$ to a quadratic refinement $e^*\mu$ on $\pi_3(W_{g,1}) = H_3(W_{g,1};\bZ)$. In \cite[Section 6.1]{KR-WFram} it was shown that under the action of $\Gamma_{g,1} := \pi_0 \Diff_\partial(W_{g,1})$ the set $\theta^{\sfr}(W_{g,1} ; \ell^{\sfr}_{\partial W_{g,1}}) := \pi_0 \Theta^{\sfr}(W_{g,1} ; \ell^{\sfr}_{\partial W_{g,1}})$ has two orbits, distinguished by the Arf invariant of the corresponding quadratic form. However, we will not need to know what the Arf invariant of $e^*\mu$ is.

\subsection{The groups $\Gamma_{g,1}^{\hyp}$ and $\Stab_{\Gamma_{g,1}}(\ell^{\hyp}_{W_{g,1}})$}

If we define 
$$\Gamma_{g,1}^{\hyp} := \pi_1(\Theta^{\hyp}(W_{g,1} ; \ell^{\hyp}_{\partial W_{g,1}}) \hcoker \Diff_\partial(W_{g,1}) , \ell^{\hyp}_{W_{g,1}})$$
then the long exact sequence of homotopy groups for this homotopy quotient has a portion
$$\pi_1(\Theta^{\hyp}(W_{g,1} ; \ell^{\hyp}_{\partial W_{g,1}}), \ell^{\hyp}_{W_{g,1}}) \lra \Gamma_{g,1}^{\hyp} \lra \Stab_{\Gamma_{g,1}}(\ell^{\hyp}_{W_{g,1}}) \lra 1,$$
where $\Stab_{\Gamma_{g,1}}(\ell^{\hyp}_{W_{g,1}})$ denotes the stabiliser of $[\ell^{\hyp}_{W_{g,1}}] \in \pi_0 \Theta^{\hyp}(W_{g,1} ;  \ell^{\hyp}_{\partial W_{g,1}})$ under the $\Gamma_{g,1}$-action. Using Lemma \ref{lem:SFReqHYP}, $\Gamma_{g,1}^{\hyp}$ is identified with the analogous stably framed mapping class group $\Gamma_{g,1}^{\sfr}$, and similarly for the stabiliser. Considered as the stabiliser of a stable framing, the latter has been determined up to an extension in \cite{KR-WFram}. Specifically, combining the discussion in Section 6.1 of that paper with its Proposition 5.1 and Section 3.4 gives a central extension
\begin{equation}\label{eq:StabExtWg}
1 \lra \Theta_7  \lra \Stab_{\Gamma_{g,1}}(\ell^{\hyp}_{W_{g,1}}) \lra \mathrm{Sp}_{2g}^{\text{q or a}}(\bZ) \lra 1,
\end{equation}
where $\mathrm{Sp}_{2g}^{\text{q or a}}(\bZ)$ denotes the stabiliser of a quadratic refinement of Arf invariant 0 or 1 respectively. (Depending on what the Arf invariant of $e^*\mu$ is.)

In \cite[Lemma 7.5]{grwabelian} there is described\footnote{Strictly speaking that source only discusses $\mathrm{Sp}_{2g}^{\text{q}}(\bZ)$, but in a stable range the cohomology of $\mathrm{Sp}_{2g}^{\text{q}}(\bZ)$ and $\mathrm{Sp}_{2g}^{\text{a}}(\bZ)$ agree.} a class $\mu \in H^2(\mathrm{Sp}_{2g}^{\text{q or a}}(\bZ);\bZ)$, characterised by two properties:
\begin{enumerate}[(i)]
\item $\mu$ vanishes when restricted to the $\bZ/4$ subgroup generated by $\big(\begin{smallmatrix}
0 & -1\\
1 & 0
\end{smallmatrix}\big) \in \mathrm{Sp}_{2}^{\text{q}}(\bZ)$,

\item the map $\mu_* : H_2(\mathrm{Sp}_{2g}^{\text{q or a}}(\bZ);\bZ) \to \bZ$ is surjective as long as $g \geq 2$ in the case $q$ and $g \geq 3$ in the case $a$\footnote{See \cite[Lemma 3.5]{KrannichMCG} for $\mathrm{Sp}_{4}^{\text{q}}(\bZ)$; for $\mathrm{Sp}_{6}^{\text{a}}(\bZ)$ observe it contains $\mathrm{Sp}_{4}^{\text{q}}(\bZ)$.}, and 8 times it is the signature map. It generates $\mathrm{Hom}(H_2(\mathrm{Sp}_{2g}^{\text{q or a}}(\bZ);\bZ), \bZ)$.
\end{enumerate}
This class $\mu$ can be multiplied by $\Sigma_\mathrm{Milnor} \in \Theta_7$ to give a class $\mu \cdot \Sigma_\mathrm{Milnor} \in H^2(\mathrm{Sp}_{2g}^{\text{q or a}}(\bZ);\Theta_7)$, corresponding to a central extension
\begin{equation}\label{eq:UCE}
0 \lra \Theta_7 \lra E \lra \mathrm{Sp}_{2g}^{\text{q or a}}(\bZ) \lra 1.
\end{equation}

\begin{theorem}
The extension \eqref{eq:StabExtWg} is isomorphic to the extension \eqref{eq:UCE}.
\end{theorem}

\begin{proof}
Including $\Stab_{\Gamma_{g,1}}(\ell^{\hyp}_{W_{g,1}})$ into $\Gamma_{g,1}$, we obtain a map of central extensions
\begin{equation*}
\begin{tikzcd}
1 \rar &\Theta_7 \arrow[d, equals] \rar & \Stab_{\Gamma_{g,1}}(\ell^{\hyp}_{W_{g,1}}) \dar{i} \rar & \mathrm{Sp}_{2g}^{\text{q or a}}(\bZ) \rar \dar{j} & 1\\
1 \rar & \Theta_7 \rar & \Gamma_{g,1} \rar & \Gamma_{g,1}/\Theta_7 \rar & 1.
\end{tikzcd}
\end{equation*}
Krannich \cite{KrannichMCG} has analysed the lower extension, as follows (we implicitly specialise his results to the case $2n=6$ without further comment).

Using the notation $\Gamma_{g,1/2} = \Gamma_{g,1}/\Theta_7$, in \cite[(1.7)]{KrannichMCG} he gives an extension\footnote{Our map $j$ splits this extension over the subgroup $\mathrm{Sp}_{2g}^{\text{q or a}}(\bZ)$, which does not contradict the fact \cite[Theorem 2.2]{KrannichMCG} that it is not split over $\mathrm{Sp}_{2g}(\bZ)$.}
$$1 \lra H^3(W_{g,1} ; S\pi_3(\SO(3))) \lra \Gamma_{g,1}/\Theta_7 \overset{p}\lra \mathrm{Sp}_{2g}(\bZ) \lra 0.$$
The set of stable framings $\pi_0\Theta^{\sfr}(W_{g,1} ; \ell^{\sfr}_{\partial W_{g,1}})$ is a $[W_{g,1}/\partial, \SO]_*$-torsor, and acting on the stable framing $\ell^{\sfr}_{W_{g,1}}$ corresponding via Lemma \ref{lem:SFReqHYP} to $\ell^{\hyp}_{W_{g,1}}$ gives a crossed homomorphism
$$s: \Gamma_{g,1} \lra [W_{g,1}/\partial, \SO]_* = H^3(W_{g,1} ; \pi_3(\SO)),$$
which descends to $\Gamma_{g,1}/\Theta_7$; this is completely parallel to the discussion in Section \ref{sec:TS}. Together the maps $s$ and $p$ give a homomorphism
$$(s,p) : \Gamma_{g,1}/\Theta_7 \lra H^3(W_{g,1} ; \pi_3(\SO)) \rtimes \mathrm{Sp}_{2g}(\bZ) = \bZ^{2g} \rtimes \mathrm{Sp}_{2g}(\bZ).$$
This is not an isomorphism, because $S\pi_3(\SO(3)) \to \pi_3(\SO)$ is not (this is related to Corollary \ref{cor:tIsMultBy2}), but it is injective. Krannich shows that $\Gamma_{g,1}/\Theta_7$ has trivial abelianisation, and then defines a cohomology class ``$\tfrac{\chi^2 - \mathrm{sgn}}{8}$'' $\in H^2(\Gamma_{g,1}/\Theta_7; \bZ)$ by showing in \cite[Lemma 3.19 (iii)]{KrannichMCG} that the composition
$$H_2(\Gamma_{g,1}/\Theta_7 ; \bZ) \overset{(s,p)_*}\lra H_2(\bZ^{2g} \rtimes \mathrm{Sp}_{2g}(\bZ) ;\bZ) \overset{\chi^2-\mathrm{sgn}}\lra \bZ$$
has image $8 \cdot \bZ$. Here $\mathrm{sgn}$ is the signature map, factoring over $H_2(\mathrm{Sp}_{2g}(\bZ); \bZ)$, and $\chi^2$ is obtained by cup-squaring the canonical twisted cohomology class $\chi \in H^1(\bZ^{2g} \rtimes \mathrm{Sp}_{2g}(\bZ)  ; \bZ^{2g})$ then applying $\lambda : \bZ^{2g} \otimes \bZ^{2g} \to \bZ$ to the coefficients. In \cite[Theorem 3.22]{KrannichMCG} (using \cite[Lemma 3.4]{KrannichMCG}) he shows that the lower extension is classified by the cohomology class
$$-\tfrac{\chi^2 - \mathrm{sgn}}{8} \cdot \Sigma_\mathrm{Milnor} \in H^2(\Gamma_{g,1}/\Theta_7 ; \Theta_7).$$

Pulling this back along $j : \mathrm{Sp}_{2g}^{\text{q or a}}(\bZ) \to \mathrm{Sp}_{2g}(\bZ) \to  \Gamma_{g,1}/\Theta_7$, and using that $\chi$ vanishes on this subgroup, it follows that the top extension is classified by $\mu \cdot \Sigma_{\mathrm{Milnor}} \in H^2(\mathrm{Sp}_{2g}^{\text{q or a}}(\bZ) ; \Theta_7)$.
\end{proof}

We record the following consequences:

\begin{corollary}\label{cor:Wg1Properties}
Suppose that $g \geq 3$.
\begin{enumerate}[(i)]
\item In the Leray--Hochschild--Serre spectral sequence for the extension \eqref{eq:StabExtWg} the differential $d_2 : H_2(\mathrm{Sp}_{2g}^{\text{q or a}}(\bZ); \bZ) \to \Theta_7$ is surjective.

\item The finite residual of $\Stab_{\Gamma_{g,1}}(\ell^{\hyp}_{W_{g,1}})$ is the central subgroup $\Theta_7$.
\end{enumerate}
\end{corollary}

\begin{proof}
Item (i) follows from the facts that $\mu : H_2(\mathrm{Sp}_{2g}^{\text{q or a}}(\bZ);\bZ) \to \bZ$ is surjective and $\Sigma_\mathrm{Milnor}$ generates $\Theta_7$. Item (ii) follows from the discussion in \cite{KrannichRW} near equation (6).
\end{proof}

\subsection{Calculating the abelianisation of $\Gamma_{g,1}^{\hyp}$}\label{sec:WgAbCalc}

The abelianisation of $\Gamma_{g,1}^{\hyp} = \Gamma_{g,1}^{\sfr}$ can be calculated in a way completely parallel to Section \ref{sec:GRW}, using $\bar{\theta}^\sfr : E\OO \to B\OO$ and its pullback $\theta^\sfr : \SO/\SO(6) \to B\SO(6)$ to an unstable tangential structure. The corresponding Thom spectrum is now $MT\theta^\sfr(6) = \Sigma^{\infty-6} \SO/\SO(6)_+$, so there is a map
$$H_1(\Gamma_{g,1}^{\hyp};\bZ) \lra \pi_1^s(\Sigma^{\infty-6} \SO/\SO(6)_+)$$
which is an isomorphism as long as $g \geq 5$. The corresponding long exact sequence simplifies to a split short exact sequence
$$0 \to \bZ/4 = \pi_7^s(\SO/\SO(6)) \to \pi_1^s(\Sigma^{\infty-6} \SO/\SO(6)_+) \to \pi_7^s(S^0) = \bZ/240\{\sigma\} \to 0.$$

\subsection{Relation to $X_d$}

The embedding $e$ gives a map of central extensions
\begin{equation*}
\begin{tikzcd}
1 \rar & \Theta_7  \rar \dar{\Phi} & \Stab_{\Gamma_{g,1}}(\ell^{\hyp}_{W_{g,1}}) \dar \rar & \mathrm{Sp}_{2g}^{\text{q or a}}(\bZ) \rar \dar & 1\\
1 \rar & \mathrm{K}_d  \rar & \Stab_{\MCG_d}(\ell^{\hyp}_{X_d}) \rar & \mathrm{Aut}(\pi_3(X_d), \lambda, \mu) \rar & 1.
\end{tikzcd}
\end{equation*}
As ``finite residual'' is covariantly functorial for group homomorphisms, Corollary \ref{cor:Wg1Properties} (ii) shows that $\mathrm{Im}(\Phi)$ is contained in the finite residual of $\Stab_{\MCG_d}(\ell^{\hyp}_{X_d})$. Combined with Theorem \ref{thm:Main} this proves Theorem \ref{thm:B}.

\begin{remark}
Using this map of central extensions, Corollary \ref{cor:Wg1Properties} (i) implies that the differential $d_2 : H_2(\mathrm{Aut};\bZ) \to \mathrm{K}_d$ in \eqref{eq:SS} hits the subgroup $\mathrm{Im}(\Phi)$; the rest of the discussion in Section \ref{sec:FurtherUpperBounds} shows that this is precisely what it hits.
\end{remark}

\begin{remark}\label{rem:WgAbApplication}
It follows from the calculation in Section \ref{sec:WgAbCalc} that if $[f] \in \MCG_d^\hyp$ can be supported on $e(W_{g,1}) \subset X_d$, then its image under
$$\MCG_d^\hyp \lra H_1(\MCG_d^\hyp;\bZ) \overset{\sim}\lra \pi_1^s(MT{\theta}^\hyp(6)) \lra \pi_7^s(MT\bar{\theta}^\hyp)$$
lies in the image of $\iota_* : \pi_7^s(S^0) \to \pi_7^s(MT\bar{\theta}^\hyp)$. This gives another point of view on the proof of Lemma \ref{lem:kappaZeroOnAlpha}.
\end{remark}

\noindent\textbf{Declarations required by Springer Nature.} On behalf of all authors, the corresponding author states that there is no conflict of interest. Being a typical paper in pure mathematics, there is no ``data'' supporting this paper. I am therefore encouraged to write: No datasets were generated or analysed during the current study.

\bibliographystyle{amsalpha}
\bibliography{biblio}

\end{document}